\documentclass[12pt]{amsart}
\usepackage{newtxtext}
\usepackage[margin=1in]{geometry}
\usepackage{amsmath,amssymb,amsthm,mathtools}
\usepackage{newtxmath}
\usepackage{hyperref}
\usepackage{microtype}
\allowdisplaybreaks

\hypersetup{
  colorlinks=true,
  linkcolor=blue,
  citecolor=blue,
  urlcolor=blue,
  pdftitle={Signed Generalized Stirling Polynomials, Nested Sums, and Hyperbolic Secant Integral Identities},
  pdfauthor={Abdulhafeez A. Abdulsalam and Michael J. Schlosser},
  pdfkeywords={Malmsten integral, hyperbolic secant integrals, signed generalized Stirling polynomials, orthogonal polynomials, nested sums, Hurwitz zeta function}
}

\numberwithin{equation}{section}

\DeclareMathOperator{\sech}{sech}
\newcommand{\F}{\mathcal F}
\newcommand{\Mal}{\mathcal M}
\newcommand{\N}{\mathbb N}
\newcommand{\dd}{\,\mathrm{d}}
\newcommand{\stirlingone}[2]{\genfrac{[}{]}{0pt}{}{#1}{#2}}
\newcommand{\rstirlingone}[3]{{\genfrac{[}{]}{0pt}{}{#1}{#2}}_{#3}}

\newtheorem{theorem}{Theorem}[section]
\newtheorem{proposition}[theorem]{Proposition}
\newtheorem{lemma}[theorem]{Lemma}
\newtheorem{corollary}[theorem]{Corollary}
\newtheorem{remark}[theorem]{Remark}

\title[Signed Stirling Polynomials and Nested Sums]{Signed Generalized Stirling Polynomials, Nested Sums, and Hyperbolic Secant Integral Identities}
\author{Abdulhafeez A. Abdulsalam}
\address{Independent researcher, Lagos, Nigeria}
\email{hafeez147258369@gmail.com}
\thanks{The first author acknowledges support from a \href{https://mathematik.univie.ac.at/third-mission/zusammenarbeit-mit-dem-globalen-sueden/}{Vienna African Scholarship} of the University of Vienna}
\author{Michael J. Schlosser}
\address{Faculty of Mathematics, University of Vienna, Oskar-Morgenstern-Platz 1, A-1090 Vienna, Austria}
\email{michael.schlosser@univie.ac.at}
\thanks{The second author's research was partly supported by FWF Austrian Science Fund grant \href{https://www.doi.org/10.55776/P32305}{10.55776/P32305}.}
\date{May 2026}

\begin{document}
\begin{abstract}
We study signed generalized Stirling polynomials $P_k(m,x)$ arising
in closed forms for Malmsten-type hyperbolic secant integrals. Their product
structure is used to prove recurrences, gamma--polygamma formulas for
$P_{m-s}(m,x)$, a central vanishing identity, a finite approximation to
$\cosh \pi x$, and a limit formula for $\pi$. We identify these polynomials
as signed residues of the equal-period Barnes multiple zeta function,
derive their reflection formula, and obtain finite parity-cancellation
and Stirling cycle-number identities. We also compare the centered
polynomials with a classical Meixner--Pollaczek orthogonal family.
We evaluate finite nested sums built from the sequence $\chi_n$.
Fixing the lower bounds turns these sums into coefficient-counting
problems: common lower bounds give binomial coefficients and staircase
bounds give Catalan numbers. Combining these counts with known formulas
for $\chi_j$ yields explicit hyperbolic-secant integral evaluations
involving Catalan's constant, zeta values, and polygamma values.
A Wolfram Language package accompanies the formulas.
\end{abstract}

\maketitle

\section{Introduction}

Malmsten-type logarithmic integrals have been studied from several points of
view.  Blagouchine traced their historical origin back to Malmsten and showed
that a broad family of such integrals can be evaluated by contour-integration
methods \cite{Blagouchine2014}. In a different direction, the first author gave a
direct proof of Malmsten's integral from its hyperbolic form and generalized
this integral by studying the sequence \cite{AbdulsalamMalmsten2025}
\[
  \Mal_n
  :=
  \int_0^\infty \log x\,\sech^n x\dd x,
  \qquad n\in\N .
\]
Here $\N=\{1,2,3,\ldots\}$.
There is also a Mellin-transform interpretation of this sequence. Write
$B(z,s)$ for the generalized Dirichlet beta function denoted by $\beta(z,s)$
in \cite[(1.1)]{Yakubovich2025}, so that
\[
  \Gamma(s)B(z,s)=\int_0^\infty x^{s-1}\sech^z x\dd x,
  \qquad \Re z>0,\quad \Re s>0.
\]
Differentiating with respect to $s$ at $s=1$ gives
\begin{equation}\label{eq:malmsten-mellin}
  \Mal_n=\partial_sB(n,1)-\gamma B(n,1),
\end{equation}
where $\gamma \approx 0.5772$ is the Euler--Mascheroni constant and
$\partial_sB(n,1)=
\left.\frac{\partial}{\partial s}B(n,s)\right|_{s=1}$. Indeed, for $s$ near $1$, the differentiated integrand is dominated near
zero by a constant multiple of $x^{-1/2}|\log x|$, and at infinity by
an exponentially decreasing function. Thus differentiation under the
integral sign is justified.

The present work takes as its starting point the polynomial structure in the
closed forms for this sequence, namely the signed generalized Stirling
polynomials of the first kind.  They are defined by the
generating relation \cite[(4)]{AbdulsalamMalmsten2025}
\begin{equation}\label{defnrelation}
  (j+1)_m
  =
  \sum_{k=0}^{m}(-1)^kP_k(m,x)(j+x)^{m-k},
\end{equation}
where $(a)_m=a(a+1)\cdots(a+m-1)$ denotes the rising factorial.  We use
$a^{\underline m}=a(a-1)\cdots(a-m+1)$ for the falling factorial. 
They admit the explicit form \cite[Lemma~1]{AbdulsalamMalmsten2025}
\begin{equation}\label{eq:P-explicit}
  P_k(m,x)
  =
  \sum_{r=0}^{k}
  (-1)^{k-r}x^r
  \binom{r+m-k}{m-k}
  \stirlingone{m+1}{r+m-k+1},
  \qquad 0\le k\le m .
\end{equation}
When $x=0$, this formula is read separately as
\[
  P_k(m,0)=(-1)^k\stirlingone{m+1}{m-k+1},
  \qquad 0\le k\le m,
\]
so no convention for $0^0$ is needed. Here $\stirlingone{n}{k}$ denotes the Stirling cycle number, or unsigned
Stirling number of the first kind.  Combinatorially, $\stirlingone{n}{k}$
counts the permutations of an $n$-element set with exactly $k$ disjoint
cycles; see \cite[p.~824, \S24.1.3(A)]{AbramowitzStegun1964}.  An explicit formula for
$\stirlingone{n}{k}$ depends on the Stirling number of the second kind
$\left\{\begin{smallmatrix}n\\k\end{smallmatrix}\right\}$, which counts
partitions of a set of $n$ labelled objects into $k$ nonempty unlabelled
subsets.
Equivalently, $k!\left\{\begin{smallmatrix}n\\k\end{smallmatrix}\right\}$
counts surjections from a set of $n$ distinct elements onto a set of $k$
distinct elements, so inclusion--exclusion gives
\cite[p.~824, \S24.1.4(C)]{AbramowitzStegun1964}
\[
  \left\{\begin{matrix}n\\k\end{matrix}\right\}
  =
  \frac1{k!}\sum_{\ell=0}^{k}
  (-1)^{k-\ell}\binom{k}{\ell}\ell^n.
\]
We use $\stirlingone{n}{0}=\left\{\begin{smallmatrix}n\\0\end{smallmatrix}\right\}=0$ and $\stirlingone{n}{n} = 1$
for $n\ge1$, and $\stirlingone{0}{0}=\left\{\begin{smallmatrix}0\\0\end{smallmatrix}\right\}=1$.
Substituting this formula into the symmetric formula expressing
$\stirlingone{n}{k}$ in terms of Stirling numbers of the second kind
$\left\{\begin{smallmatrix}n\\k\end{smallmatrix}\right\}$
\cite[p.~824, \S24.1.3(C)]{AbramowitzStegun1964} yields
\begin{equation}\label{twosumformula}
  \stirlingone{n}{k}
  =
  \sum_{j=n}^{2n-k}
  \binom{j-1}{k-1}\binom{2n-k}{j}
  \sum_{m=0}^{j-n}
  \frac{(-1)^{m+n-k}m^{j-k}}{m!(j-n-m)!}.
\end{equation}
Thus the Stirling cycle numbers admit the two-sum expression
\eqref{twosumformula}; we are not aware of a comparable finite one-sum formula.

The polynomials $P_k(m,x)$ appear in \cite{Ripon2015} in connection with a generalization of certain logarithmic integrals. However, as observed by the first author in \cite[Remark 2]{AbdulsalamMalmsten2025}, the formula given there does not recover the signed generalized Stirling polynomials in general. The expression \eqref{eq:P-explicit} was derived by identifying the coefficients in the defining relation \eqref{defnrelation} with the corresponding coefficients in the generating formula for generalized Stirling polynomials, and then applying Adamchik's formula in terms of Stirling cycle numbers; see \cite[Lemma~1]{AbdulsalamMalmsten2025}. Comparing Ripon's expression for $P_l(m,x)$ in \cite[(4.3)]{Ripon2015} with \eqref{eq:P-explicit}, after replacing $r$ by $k-r$, shows that the factor $\binom{m-k}{m-l}$ is missing from the summand.

The source of the discrepancy appears to be combinatorial.  In the expansion of $P_l(m,x)$ in \cite[(4.2)]{Ripon2015}, a fixed product of $k$ distinct integers from $\{1,\ldots,m\}$ occurs with multiplicity.  After these $k$ integers are fixed, the remaining $l-k$ factors may be chosen in $\binom{m-k}{l-k}=\binom{m-k}{m-l}$ ways.  Thus the summand involving $R(m,k)$ in \cite[(4.3)]{Ripon2015} must include this extra binomial factor. This accounts for the discrepancy between Ripon's formula and \eqref{eq:P-explicit};
in particular, only the first two terms of Ripon's formula agree with
\eqref{eq:P-explicit}.  We therefore use \eqref{eq:P-explicit} for the
polynomials defined by \eqref{defnrelation}.

Setting $k=m$ in \eqref{eq:P-explicit} gives
\[
  P_m(m,x)
  =
  \sum_{r=0}^{m}
  (-1)^{r-m}
  \stirlingone{m+1}{r+1}x^r .
\]
With the same convention as in \cite[(3)]{AbdulsalamMalmsten2025}, the
generating formula for the Stirling cycle numbers \cite[(20.1)]{Connon2008}
gives
\[
  x(x-1)\cdots(x-m)
  =
  \sum_{r=0}^{m+1}
  (-1)^{m+1-r}\stirlingone{m+1}{r}x^r.
\]
Division by $x$ gives the special case
\begin{equation}\label{eq:P-diagonal}
  P_m(m,x)=(x-1)(x-2)\cdots(x-m).
\end{equation}
Equivalently, using falling-factorial notation,
$P_m(m,x)=(x-1)^{\underline{m}}$.  This identity shows that the falling factorials are recovered when the two
indices are equal. In Proposition~\ref{prop:residue-reflection}, we use the same
polynomial structure to prove the reflection formula
\[
  P_k(m,m+1-x)=(-1)^kP_k(m,x),
  \qquad 0\le k\le m.
\]
This symmetry also leads to finite parity-cancellation relations; after showing
where full finite orthogonality breaks down, we compare the centered 
polynomials with a natural Meixner--Pollaczek orthogonal family.
We recall a gamma--digamma formula from
\cite[(53)]{AbdulsalamMalmsten2025}:
\begin{equation}\label{digammap}
  P_m(m+1,x)
  =
  \frac{\Gamma(x)}{\Gamma(x-m-1)}
  \left(\psi(x)-\psi(x-m-1)\right),
\end{equation}
where $\Gamma$ is the gamma function and $\psi(x)=\Gamma'(x)/\Gamma(x)$ 
is the digamma function. In Proposition~\ref{prop:near-diagonal-gamma}, we extend this to the general formula
\begin{equation}\label{gammapgamma}
  P_{M-s}(M,x)
  =
  \frac{\Gamma(x)}{\Gamma(x-M)}
  \sum_{\mu\vdash s}
  \frac1{z_\mu}
  \prod_{j\geq1}
  \left(
  \frac{\psi_{j-1}(x)-\psi_{j-1}(x-M)}{(j-1)!}
  \right)^{m_j},
\end{equation}
where the sum is over partitions $\mu$ of $s$, and
$\mu=(1^{m_1}2^{m_2}\cdots)$ means that the part $j$ occurs $m_j$ times;
and $z_\mu$ is defined by
$z_\mu=\prod_{j\geq1} j^{m_j}m_j!$.  Here
$\psi_j$ denotes the polygamma function, with $\psi_0=\psi$.
Taking $s=1$ and $M=m+1$ recovers
\eqref{digammap}.

This same polynomial structure also leads to an approximation to
$\cosh \pi x$.  In particular, Proposition \ref{prop:cosh-product} shows that,
with $i=\sqrt{-1}$,
\[
  \cosh \pi x
  =
  \lim_{M\to\infty}
  \frac{
  \sum_{k=0}^{M}
  P_{2k}\left(2M,M+\frac12\right)(ix)^{2M-2k}
  }
  {P_{2M}\left(2M,M+\frac12\right)}.
\]
It also gives the square-root formula
\[
  \pi
  =
  \lim_{M\to\infty}
  \left[
  -2
  \frac{
  P_{2M-2}\left(2M,M+\frac12\right)
  }
  {
  P_{2M}\left(2M,M+\frac12\right)
  }
  \right]^{1/2}.
\]

The polynomials $P_k(m, x)$ also arise naturally from Barnes' zeta function.
For $n\ge1$, the equal-period Barnes multiple zeta function \cite[(1.1)]{Ruijsenaars2000} is
\[
  \zeta_n(s,x)
  =
  \sum_{r_1,\ldots,r_n\ge0}
  \frac1{(r_1+\cdots+r_n+x)^s},
  \qquad \Re s>n,\quad \Re x>0,
\]
with its meromorphic continuation.  In Proposition
\ref{prop:barnes-residue-characterization}, Adamchik's finite reduction for $\zeta_n(s,x)$
is used to compute the equal-period residues and gives the interpretation
\[
  P_k(m,x)
  =
  (-1)^k m!\operatorname*{Res}_{s=m+1-k}\zeta_{m+1}(s,x).
\]
Thus the identities for $P_k(m,x)$ may also be read as identities for the residues of Barnes' zeta function. They also give explicit identities for Stirling cycle numbers.  For
example, the gamma--polygamma formula \eqref{gammapgamma} at integer point $x=M+1$ gives the
weighted binomial transform
\[
  \sum_{r=0}^{M-s}
  (-1)^{r-M+s}
  (M+1)^r
  \binom{r+s}{s}
  \stirlingone{M+1}{r+s+1}
  =
  \stirlingone{M+1}{s+1},
  \qquad 0\le s\le M.
\]

The usefulness
of these polynomials is already visible in the closed form for
the generalized Malmsten sequence.  If
\[
  \Mal_n(a,b):=\int_0^\infty \log(ax)\sech^n(bx)\dd x,
  \qquad a>0,\quad b>0,
\]
then \cite[Theorem~1]{AbdulsalamMalmsten2025} gives, for $n\in \N$,
\begin{equation}\label{eq:malmsten-closed-form}
\begin{aligned}
  \Mal_n(a,b)
  ={}&
  \frac{2^{n-2}\Gamma^2(n/2)}{b(n-1)!}\log\frac{a}{b}
  \\
  &+
  \frac{2^{2n-1}}{b(n-1)!}
  \sum_{k=2}^{n+1}
  \left(-\frac12\right)^k
  P_{k-2}\!\left(n-1,\frac n2\right)
  \\
  &\qquad\qquad\times
  \left[
  \zeta'\!\left(k-n,\frac n4\right)
  -
  \zeta'\!\left(k-n,\frac{n+2}{4}\right)
  \right.
  \\
  &\qquad\qquad\qquad\left.
  -(\gamma+\log 4)
  \left\{
  \zeta\!\left(k-n,\frac n4\right)
  -
  \zeta\!\left(k-n,\frac{n+2}{4}\right)
  \right\}
  \right].
\end{aligned}
\end{equation}
The same change of variable $u=bx$ in the Mellin representation gives
\begin{equation}\label{eq:malmsten-scaled-mellin}
  \Mal_n(a,b)=\frac1b\left[
  \partial_sB(n,1)+\left(\log\frac ab-\gamma\right)B(n,1)\right].
\end{equation}

We use standard notation for the special functions: for $\nu\ge0$,
$\psi_\nu(z)=\dd ^\nu\psi(z)/\dd z^\nu$ denotes the polygamma function, with
$\psi_0(z)=\psi(z)=\Gamma'(z)/\Gamma(z)$; $\gamma
=\lim_{n\to\infty}\bigl(\sum_{k=1}^{n}1/k-\log n\bigr)\approx0.5772$ is the
Euler--Mascheroni constant; and
$G=\sum_{k=0}^{\infty}(-1)^k/(2k+1)^2\approx0.9159$ is Catalan's constant.
If $\mu$ is a partition of $s$, we write $\mu\vdash s$.  Writing
$\mu=(1^{m_1}2^{m_2}\cdots)$ means that the part $j$ occurs $m_j$ times
in $\mu$; only finitely many of the $m_j$ are nonzero.  We also put
$z_\mu=\prod_{j\geq1} j^{m_j}m_j!$.
For the zeta functions, $\zeta(s,u)$ denotes the Hurwitz zeta function,
initially defined for $\Re s>1$ by
$\zeta(s,u)=\sum_{k=0}^{\infty}(k+u)^{-s}$ and then by meromorphic
continuation.  We also write $\zeta(s)=\zeta(s,1)$ for the Riemann zeta
function, and $\zeta'(s,u)$ for differentiation of $\zeta(s,u)$ with respect to
$s$; see \cite[\S\S5.2, 5.15, 25.2, 25.11]{DLMF}. The ordinary Dirichlet beta function is defined by
$\beta(s)=\sum_{k=0}^{\infty}(-1)^k/(2k+1)^s$ for $\Re s>0$
and extended to the whole complex plane by analytic continuation.
In particular, $\beta(2)=G$.
This is the computational starting point. Once \eqref{eq:malmsten-closed-form}
is available, the related sequences are obtained from the recurrence relations
below together with the closed form for $\chi_n$. 

The same paper introduces three integral sequences associated with
$\Mal_n$, namely \cite[Definition~1]{AbdulsalamMalmsten2025}
\begin{align}
  \lambda_n&:=\int_0^\infty\frac{\tanh x\,\sech^n x}{x}\dd x, \notag\\
  \delta_n&:=\int_0^\infty\frac{1-\sech x}{x^2}\,\sech^n x\dd x, \notag\\
  \chi_n&:=\int_0^\infty\frac{\sech x-\sech^n x}{x^2}\dd x,
  \qquad n\in\N . \label{eq:chi-def}
\end{align}
The identities in \cite[(7), (8), (9), (10)]{AbdulsalamMalmsten2025} give, in particular,
\[
(n-1)\Mal_n = (n-2)\Mal_{n-2} - \lambda_{n-2}, \qquad n \geq 3,
\]
\begin{equation}\label{iden1hh3erf}
  n\lambda_n=\delta_{n-1}+(n-1)\lambda_{n-1},
  \qquad n\ge2,
\end{equation}
\begin{equation}\label{iden2hh4erf}
  \chi_n=\sum_{r=1}^{n-1}\delta_r,
  \qquad
  n\lambda_n=\chi_n+\lambda_1,\qquad n\ge2,
\end{equation}
and also
\[
(n-2)(n-1)\Mal_n = (n-2)^2\Mal_{n-2} - \chi_{n-2} - \lambda_1, \qquad n \geq 3.
\]
Comparing identities \eqref{iden1hh3erf} and \eqref{iden2hh4erf}, we obtain
\[
  \chi_n+\lambda_1
  =
  \delta_{n-1}+\chi_{n-1}+\lambda_1,
\]
where $\lambda_1 = \delta_0 = 4G/\pi$ and $\chi_2 = \delta_1$. Hence
\begin{equation}\label{eq:delta-chi-difference}
  \delta_n=\chi_{n+1}-\chi_n,
  \qquad n\in\N .
\end{equation}

The present work records what happens when the nested sum in
\cite[(56)]{AbdulsalamMalmsten2025},
\begin{equation}\label{eq:nested-sum}
  S_N(l_1,\ldots,l_N)
  :=
  \sum_{k_N=l_N}^{N}
  \sum_{k_{N-1}=l_{N-1}}^{k_N}
  \cdots
  \sum_{k_2=l_2}^{k_3}
  \sum_{k_1=l_1}^{k_2}
  \chi_{k_1},
\end{equation}
is specialized by fixing its lower bounds.  After these lower bounds are fixed,
the sum becomes a finite linear combination of the quantities $\chi_j$.
Substituting \eqref{eq:chi-def} then produces a single integral whose numerator
is a polynomial in $\sech x$.  The right-hand side is obtained by applying the
same linear combination to the closed forms for the individual $\chi_j$'s. 
A sum with lower limit greater than its upper limit is interpreted as zero.
We put $\F_1=0$, in agreement with $\chi_1=0$.  For $j\ge2$, we denote the
closed form for $\chi_j$ supplied by the right-hand side of
\cite[(56)]{AbdulsalamMalmsten2025} by $\F_j$, so that
\begin{equation}\label{eq:Fj}
\begin{aligned}
\F_j={}&-\frac{4G}{\pi}
+\frac{2^{2j-3}j^2}{(j-1)!}
\sum_{m=0}^{j-1}
\left(-\frac12\right)^m
P_m\!\left(j-1,\frac{j}{2}\right)
\\
&\qquad\qquad\times
\left[
\zeta'\!\left(m-j+2,\frac{j}{4}\right)
-\zeta'\!\left(m-j+2,\frac{j+2}{4}\right)
\right]
\\
&-\frac{2^{2j+1}}{(j-1)!}
\sum_{m=0}^{j+1}
\left(-\frac12\right)^m
P_m\!\left(j+1,\frac{j+2}{2}\right)
\\
&\qquad\qquad\times
\left[
\zeta'\!\left(m-j,\frac{j+2}{4}\right)
-\zeta'\!\left(m-j,\frac{j+4}{4}\right)
\right].
\end{aligned}
\end{equation}
The sequence $\lambda_n$ is also contained in the integral evaluations of
Yakubovich \cite[(4.29), (4.38)]{Yakubovich2025}. In the notation of that
paper, put
\[
  h_j=\int_0^\infty\frac{\sinh x}{x\cosh^j x}\dd x,
  \qquad j\ge2.
\]
Since $\tanh x\,\sech^n x=\sinh x/\cosh^{n+1}x$, we have
\begin{equation}\label{eq:yakubovich-sequence}
  \lambda_n=h_{n+1},\qquad
  \chi_n=\F_n=nh_{n+1}-\frac{4G}{\pi},\qquad n\ge1.
\end{equation}
Here the second equality follows from \eqref{iden2hh4erf} and
$\lambda_1=4G/\pi$. Consequently,
Yakubovich's evaluations give another finite expression for $\F_n$;
we record it in \eqref{eq:F-beta-finite} and \eqref{eq:F-zeta-finite} below.
The lower-bound conditions in the present work determine the multiplicities
with which these individual evaluations occur in the nested sums.

The Hurwitz-zeta derivative differences in \eqref{eq:Fj} have the half-shift
form treated in \cite[Lemma 3]{AbdulsalamMalmsten2025}.  The only point
requiring care is the case where the first argument is $1$, as the two separate
terms then have the same pole, but their difference has a finite value.  In the
first sum this happens when $m=j-1$, and in the second sum it happens when $m = j + 1$. This
finite part is read from the Laurent expansion of the Hurwitz zeta function at
$s=1$; see, for example, \cite{Berndt1972}:
\[
  \zeta(s,a)
  =
  \frac1{s-1}
  +
  \sum_{n=0}^{\infty}
  \frac{(-1)^n}{n!}\gamma_n(a)(s-1)^n,
\]
which defines the generalized Stieltjes constants $\gamma_n(a)$.
Differentiating the difference of two such expansions shows that the pole terms
cancel; hence, in the case $s=1$,
\[
  \left[\zeta'(s,a)-\zeta'(s,b)\right]_{s=1}
  =
  -\gamma_1(a)+\gamma_1(b).
\]
In particular, the $n=1$ case of \cite[(23)]{AbdulsalamMalmsten2025} gives
\begin{align*}
  \zeta'\!\left(1,\frac{m}{2}\right)
  -
  \zeta'\!\left(1,\frac{m+1}{2}\right)
  ={}&
  -\log 2
  \left[
  \psi\!\left(\frac{m}{2}\right)
  -
  \psi\!\left(\frac{m+1}{2}\right)
  \right]
  \\
  &+
  (-1)^m\left(\log^2 2-2\gamma\log 2\right)
  +
  2(-1)^m(1-\delta_{m1})
  \sum_{r=1}^{m-1}\frac{(-1)^r\log r}{r},
\end{align*}
where $\delta_{m1}$ is the Kronecker delta and $\gamma$ is the Euler--Mascheroni constant. The half-shift form also illustrates why the Malmsten integral $\Mal_1$ is useful.  In the earlier paper
\cite[(38)]{AbdulsalamMalmsten2025}, one deduces from $\Mal_1$ that
\[
  \lim_{n\to1}
  \left[
  \zeta'\!\left(n,\frac{1}{4}\right)
  -
  \zeta'\!\left(n,\frac{3}{4}\right)
  \right]
  =
  2\pi\log\left(
  \frac{\Gamma\left(\frac{3}{4}\right)\sqrt{2\pi}}
       {\Gamma\left(\frac{1}{4}\right)}
  \right)
  +\pi\left(\gamma+\log 4\right).
\]
Thus the Malmsten integral evaluation connects pole-cancelled
Hurwitz-zeta derivative differences with known constants at quarter arguments. 
The present work connects the polynomials $P_k(m,x)$ with elementary
symmetric functions, expansions indexed by integer partitions, Stirling
cycle-number identities, and residues of the equal-period Barnes zeta
function. It also relates the centered polynomial products to
Meixner--Pollaczek polynomials and develops a coefficient-counting
approach to nested sums of hyperbolic integral evaluations, yielding
identities involving binomial coefficients and Catalan numbers.

The last section returns to this coefficient viewpoint, giving a recurrence for
arbitrary lower bounds and considering normalized limits of the two main
families.  This yields
\[
  \int_0^\infty
  \frac{\sech x-\dfrac{\sech^m x}{2-\sech x}}{x^2}\dd x
  =
  \sum_{r=0}^{\infty}2^{-r-1}\F_{m+r},
\]
for fixed $m\ge1$, and
\[
  \int_0^\infty
  \frac{\sech x-\dfrac{\sech x}{(2-\sech x)^2}}{x^2}\dd x
  =
  \sum_{j=1}^{\infty}\frac{j}{2^{j+1}}\F_j .
\]

The results of the present work are organized around two main strands.
First, Lemma~\ref{lem:central-vanishing} provides a more transparent
proof of the central vanishing identity, which is subsequently used
in the $\cosh$ approximation of Proposition~\ref{prop:cosh-product},
the residue evaluation for the equal-period Barnes zeta function,
and an explicit identity for Stirling cycle numbers.
Proposition~\ref{prop:P-recurrences} gives recurrence relations for
the polynomials, while Proposition~\ref{prop:near-diagonal-gamma}
gives gamma--polygamma formulas. The Barnes-zeta application
interprets the polynomials as signed residues and records the
corresponding differentiation and reflection formulas. We then
derive explicit identities for Stirling cycle numbers.

Second, Lemma~\ref{lem:counting} reduces the nested sums to
coefficient counting. Combining these counts with the individual
integral evaluations gives the common-lower-bound identity in
Proposition~\ref{prop:common-lower-bound-identity} and the staircase
identity in Proposition~\ref{prop:staircase-identity}, both summarized
in Theorem~\ref{thm:main-nested-identities}. 
The individual integrals occurring in these nested sums can
also be evaluated using Yakubovich's formulas \cite{Yakubovich2025}.

Section~\ref{sec:further-directions} gives a general recurrence
for arbitrary lower bounds and normalized limiting identities
for the two main families.
Subsection~\ref{subsec:orthogonality-reflection} examines parity
cancellation and the limits of full finite orthogonality for the
centered polynomials, and develops their connection with central
factorial coefficients and Meixner--Pollaczek polynomials.
The subsequent subsections discuss further limiting families
and analytic interpolation, then use Yakubovich's evaluations
to obtain finite formulas involving Dirichlet beta and Riemann
zeta values. These formulas explain the arithmetic patterns in
the examples and give rational-span inclusions for nested sums
of every finite depth.

For reproducibility, the accompanying Wolfram Language package
is described in Section~\ref{sec:package}; it implements the
coefficient counts and the closed-form reductions used in
the examples.

\begin{theorem}[main nested-sum identities]\label{thm:main-nested-identities}
Let $N\ge1$.  If $1\le m\le N$ and $l_1=\cdots=l_N=m$, then
\[
  \int_0^\infty
  \left\{
  \binom{2N-m}{N-m}\sech x
  -
  \sum_{j=m}^{N}\binom{2N-j-1}{N-j}\sech^j x
  \right\}\frac{\dd x}{x^2}
  =
  \sum_{j=m}^{N}\binom{2N-j-1}{N-j}\F_j .
\]
If $(l_1,\ldots,l_N)=(1,2,\ldots,N)$, then
\[
  \int_0^\infty
  \left\{
  C_N\sech x
  -
  \sum_{j=1}^{N}
  \frac{j}{N}
  \binom{2N-j-1}{N-1}\sech^j x
  \right\}\frac{\dd x}{x^2}
  =
  \sum_{j=1}^{N}
  \frac{j}{N}
  \binom{2N-j-1}{N-1}\F_j,
\]
where $C_N$ denotes the $N$th Catalan number.
\end{theorem}

\begin{proof}
These are proved in Propositions~\ref{prop:common-lower-bound-identity} and
\ref{prop:staircase-identity}, respectively.
\end{proof}

\section{Results}

We now collect the main consequences of the polynomial structure. 
We begin with a central vanishing identity for the signed generalized Stirling polynomials, 
since it is useful both for computation and for the finite approximation to $\cosh \pi x$ that follows.  
We then give recurrence relations and gamma--polygamma forms for these
polynomials, and interpret them as signed residues of the equal-period Barnes
multiple zeta function. We also record the resulting identities for Stirling cycle numbers.
The rest of the section turns to the nested sums.
After reducing them to explicit
coefficient counts, we combine those counts with the closed forms for the
individual $\chi_j$'s.  This gives, in particular, the common-lower-bound
family and the staircase family of evaluations.

We first establish the central vanishing identity for the signed generalized Stirling
polynomials. The identity was first obtained in
\cite[Lemma~2]{AbdulsalamMalmsten2025} from a limit formula involving the
$(n-1)$st derivative of a power of a logarithm.  In that proof, Leibniz' rule
was applied to this derivative, producing a rather complicated finite sum; the
vanishing is therefore not easy to see directly from the expression.  The defining equation 
gives a clearer reason for the cancellation. At the 
central point $x=n/2$, the structure of the defining product forces the 
relevant odd coefficients to vanish. In computer algebra implementations of
the closed forms for the integral sequences, every term containing
$P_{2r-1}\!\left(n-1,\frac n2\right)$ can therefore be discarded before
symbolic evaluation, rather than expanded with the remaining terms in a
formula.

\begin{lemma}[central vanishing]\label{lem:central-vanishing}
For $n\ge1$ and $0\le k\le n-1$,
\[
  P_k\!\left(n-1,\frac n2\right)=0
  \qquad\text{whenever }k\text{ is odd}.
\]
In particular, for $r\in\N$ and $n\ge2r$,
\[
  P_{2r-1}\!\left(n-1,\frac n2\right)
  =0.
\]
\end{lemma}

\begin{proof}
Put $m=n-1$ and $x=n/2$ in the product form obtained from
\eqref{defnrelation}, after putting $y=j+x$ and using the resulting polynomial
identity in $y$:
\[
  \prod_{r=1}^{n-1}\left(y+r-\frac n2\right)
  =
  \sum_{k=0}^{n-1}
  (-1)^kP_k\!\left(n-1,\frac n2\right)y^{n-1-k}.
\]
The factors occur in opposite pairs.  More precisely, the factor with index
$r$ pairs with the factor with index $n-r$, and
\[
  \left(y+r-\frac n2\right)
  \left(y+n-r-\frac n2\right)
  =
  y^2-\left(r-\frac n2\right)^2 .
\]
If $n$ is odd, all factors are paired and the product is a polynomial in
$y^2$.  If $n$ is even, the middle factor is $y$, and the product is $y$ times
a polynomial in $y^2$.  In both cases the parity of the powers of $y$ in the
product is the same as the parity of $n-1$.  Thus the coefficient of
$y^{n-1-k}$ vanishes whenever $k$ is odd.  Taking $k=2r-1$ gives the stated
special case.
\end{proof}

\subsection{A Cosh Approximation from the Polynomials}

\begin{proposition}\label{prop:cosh-product}
For $M\in\N$,
\begin{equation}\label{eq:finite-cosh-product}
  \frac{
  \sum_{k=0}^{M}
  P_{2k}\left(2M,M+\frac12\right)(ix)^{2M-2k}
  }
  {P_{2M}\left(2M,M+\frac12\right)}
  =
  \prod_{r=1}^{M}
  \left(1+\frac{4x^2}{(2r-1)^2}\right).
\end{equation}
Consequently,
\begin{equation}\label{eq:cosh-polynomial-limit}
  \cosh \pi x
  =
  \lim_{M\to\infty}
  \frac{
  \sum_{k=0}^{M}
  P_{2k}\left(2M,M+\frac12\right)(ix)^{2M-2k}
  }
  {P_{2M}\left(2M,M+\frac12\right)}.
\end{equation}
Moreover, for each fixed $s\ge1$,
\begin{equation}\label{eq:pi-polynomial-limit}
  \frac{\pi^{2s}}{(2s)!}
  =
  \lim_{M\to\infty}
  (-1)^s
  \frac{
  P_{2M-2s}\left(2M,M+\frac12\right)
  }
  {
  P_{2M}\left(2M,M+\frac12\right)
  }.
\end{equation}
In particular,
\begin{equation}\label{eq:pi-square-polynomial-limit}
  \pi
  =
  \lim_{M\to\infty}
  \left[
  -2
  \frac{
  P_{2M-2}\left(2M,M+\frac12\right)
  }
  {
  P_{2M}\left(2M,M+\frac12\right)
  }
  \right]^{1/2}.
\end{equation}
\end{proposition}

\begin{proof}
Putting $n=2M+1$ in \cite[Remark~3]{AbdulsalamMalmsten2025} gives
\[
  \prod_{r=1}^{M}
  \left(z^2-\frac{(2r-1)^2}{4}\right)
  =
  \sum_{k=0}^{M}
  P_{2k}\left(2M,M+\frac12\right)z^{2M-2k}.
\]
Substituting $z=ix$, the left-hand side becomes
\begin{equation}\label{eq:substituted-product}
  \prod_{r=1}^{M}
  \left(-x^2-\frac{(2r-1)^2}{4}\right)
  =
  (-1)^M
  \prod_{r=1}^{M}\frac{(2r-1)^2}{4}
  \prod_{r=1}^{M}
  \left(1+\frac{4x^2}{(2r-1)^2}\right).
\end{equation}
On the other hand, the constant term of
\[
  \sum_{k=0}^{M}
  P_{2k}\left(2M,M+\frac12\right)(ix)^{2M-2k}
\]
is $P_{2M}\left(2M,M+\frac12\right)$.  Hence
\begin{equation}
  P_{2M}\left(2M,M+\frac12\right)
  =
  (-1)^M\prod_{r=1}^{M}\frac{(2r-1)^2}{4}.
\end{equation}
Using \eqref{eq:substituted-product} and normalizing by this constant term gives
\eqref{eq:finite-cosh-product}.  Letting $M\to\infty$, and using
the product formula \cite[(4.36.2)]{DLMF}
\[
  \cosh \pi x
  =
  \prod_{r=1}^{\infty}
  \left(1+\frac{4x^2}{(2r-1)^2}\right),
\]
gives \eqref{eq:cosh-polynomial-limit}.  It remains to read off the
coefficients.  For fixed $s\ge1$, the coefficient of $x^{2s}$ in the finite
polynomial on the right-hand side of \eqref{eq:cosh-polynomial-limit} is
\[
  (-1)^s
  \frac{
  P_{2M-2s}\left(2M,M+\frac12\right)
  }
  {
  P_{2M}\left(2M,M+\frac12\right)
  },
\]
provided $M\ge s$.  Equivalently, this is the coefficient of $x^{2s}$ in the
finite product
\[
  \prod_{r=1}^{M}
  \left(1+\frac{4x^2}{(2r-1)^2}\right).
\]
For fixed $s$, these coefficients converge as $M\to\infty$, since
$\sum_{r\ge1}4/(2r-1)^2$ is convergent.  Hence the limiting coefficient is
the coefficient of $x^{2s}$ in the infinite product, and by the product formula
above this is the coefficient of $x^{2s}$ in
\[
  \cosh \pi x
  =
  \sum_{\ell=0}^{\infty}\frac{\pi^{2\ell}x^{2\ell}}{(2\ell)!}.
\]
This coefficient comparison gives \eqref{eq:pi-polynomial-limit}; the special
case $s=1$ gives
\eqref{eq:pi-square-polynomial-limit}.
\end{proof}

\begin{remark}
The expression under the square root in \eqref{eq:pi-square-polynomial-limit} is positive for each $M$.  Indeed, the
constant term computation in the proof shows that
\[
  \operatorname{sgn}
  P_{2M}\left(2M,M+\frac12\right)=(-1)^M.
\]
Comparing the coefficient of $x^2$ in \eqref{eq:substituted-product} similarly
shows that
\[
  \operatorname{sgn}
  P_{2M-2}\left(2M,M+\frac12\right)=(-1)^{M-1}.
\]
Therefore
\[
  -2
  \frac{
  P_{2M-2}\left(2M,M+\frac12\right)
  }
  {
  P_{2M}\left(2M,M+\frac12\right)
  }>0
\]
for every $M$.
\end{remark}
\subsection{Recurrences for the Polynomials}

The defining relation also gives recurrences for the polynomials
$P_k(m,x)$.  These are helpful both for computation and for seeing how the
Stirling-cycle structure enters the formulas.

\begin{proposition}\label{prop:P-recurrences}
Let $m\ge1$.  With the conventions
\begin{equation}\label{eq:refh1}
  P_0(0,x)=1,\qquad P_k(m,x)=0\quad(k<0\text{ or }k>m),
\end{equation}
one has
\begin{equation}\label{eq:P-step-recurrence}
  P_k(m,x)
  =
  P_k(m-1,x)+(x-m)P_{k-1}(m-1,x),
  \qquad 0\le k\le m .
\end{equation}
Moreover,
\begin{equation}\label{eq:P-newton-recurrence}
  kP_k(m,x)
  =
  -
  \sum_{r=1}^{k}
  \left(\sum_{s=1}^{m}(s-x)^r\right)P_{k-r}(m,x),
  \qquad 1\le k\le m .
\end{equation}
\end{proposition}

\begin{proof}
Put $y=j+x$ in the defining relation \eqref{defnrelation}.  Then
\begin{equation}\label{eq:P-product-expansion}
  \prod_{i=1}^{m}(y+i-x)
  =
  \sum_{k=0}^{m}(-1)^kP_k(m,x)y^{m-k}.
\end{equation}
Since
\[
  \prod_{i=1}^{m}(y+i-x)
  =
  (y+m-x)\prod_{i=1}^{m-1}(y+i-x),
\]
comparison of the coefficient of $y^{m-k}$ gives
\eqref{eq:P-step-recurrence}.

For \eqref{eq:P-newton-recurrence}, write
\[
  \prod_{i=1}^{m}(y+i-x)
  =
  \sum_{k=0}^{m}e_k(1-x,2-x,\ldots,m-x)y^{m-k}.
\]
Here $e_k$ denotes the elementary symmetric polynomial; see \cite[Ch.~I, (2.2), p.~19]{Macdonald1995}.  Thus
\begin{equation}\label{eq:P-elementary-symmetric}
  P_k(m,x)=(-1)^ke_k(1-x,2-x,\ldots,m-x).
\end{equation}
Newton's identities
\cite[Ch.~I, (2.11'), p.~23]{Macdonald1995} give
\[
  ke_k
  =
  \sum_{r=1}^{k}(-1)^{r-1}p_r e_{k-r},
  \qquad
  p_r=\sum_{s=1}^{m}(s-x)^r,
\]
which is precisely \eqref{eq:P-newton-recurrence}.
\end{proof}

\begin{remark}
The recurrence \eqref{eq:P-step-recurrence} may also be recovered directly from
the usual recurrence for the Stirling cycle numbers \cite[(26.8.18)]{DLMF},
\begin{equation}\label{stirrlhaf1}
  \stirlingone{n}{k}
  =
  \stirlingone{n-1}{k-1}
  +(n-1)\stirlingone{n-1}{k}.
\end{equation}
Indeed, applying this relation to the Stirling cycle number in
\eqref{eq:P-explicit} gives, for $k<m$,
\[
  P_k(m,x)
  =
  P_k(m-1,x)
  +
  \frac{x}{m-k}\frac{\dd }{\dd x}P_k(m-1,x)
  -
  mP_{k-1}(m-1,x).
\]
The derivative term comes from differentiating the defining equation
\eqref{defnrelation} with respect to $x$, which gives
\begin{equation}\label{deriveqnhaf1}
  \frac{\dd}{\dd x}P_k(m-1,x)
  =
  (m-k)P_{k-1}(m-1,x),
\end{equation}
and therefore the preceding identity reduces to
\eqref{eq:P-step-recurrence}.
\end{remark}

\begin{proposition}[gamma--polygamma formulas]
\label{prop:near-diagonal-gamma}
Let $M\ge1$ and $0\le s\le M$.  For $x\notin\{1,2,\ldots,M\}$,
\begin{equation}\label{eq:near-diagonal-esym}
  P_{M-s}(M,x)
  =
  \frac{\Gamma(x)}{\Gamma(x-M)}
  e_s\!\left(
  \frac1{x-1},\frac1{x-2},\ldots,\frac1{x-M}
  \right).
\end{equation}
Equivalently, writing each partition of $s$ as
$\mu=(1^{m_1}2^{m_2}\cdots)$ and putting
$z_\mu=\prod_{j\geq1} j^{m_j}m_j!$,
\begin{equation}\label{eq:near-diagonal-partition}
  P_{M-s}(M,x)
  =
  \frac{\Gamma(x)}{\Gamma(x-M)}
  \sum_{\mu\vdash s}
  \frac1{z_\mu}
  \prod_{j\geq1}
  \left(
  \frac{\psi_{j-1}(x)-\psi_{j-1}(x-M)}{(j-1)!}
  \right)^{m_j}.
\end{equation}
At $x\in\{1,\ldots,M\}$, the identities are understood by taking the removable
limit on the right-hand side.
In particular,
\begin{equation}\label{eq:near-diagonal-first}
  P_{M-1}(M,x)
  =
  \frac{\Gamma(x)}{\Gamma(x-M)}
  \left[\psi(x)-\psi(x-M)\right],
\end{equation}
and, for $M\ge2$,
\begin{equation}\label{eq:near-diagonal-second}
  P_{M-2}(M,x)
  =
  \frac{\Gamma(x)}{\Gamma(x-M)}
  \frac12
  \left[
  \left(\psi(x)-\psi(x-M)\right)^2
  +
  \psi_1(x)-\psi_1(x-M)
  \right].
\end{equation}
Thus, taking $M=m+1$ in \eqref{eq:near-diagonal-first} recovers the formula
\eqref{digammap}.
\end{proposition}

\begin{proof}
By \eqref{eq:P-elementary-symmetric},
\begin{equation}\label{eqaftwopnine}
  P_k(M,x)
  =
  (-1)^ke_k(1-x,2-x,\ldots,M-x)
  =
  e_k(x-1,x-2,\ldots,x-M),
\end{equation}
because $e_k$ is homogeneous of degree $k$ (replacing each entry $r-x$ by
$-(x-r)$ contributes the factor $(-1)^k$).  Now take $k=M-s$ and use the 
elementary-symmetric identity
\[
  e_{M-s}(a_1,\ldots,a_M)
  =
  a_1a_2\cdots a_M\,
  e_s\!\left(\frac1{a_1},\ldots,\frac1{a_M}\right),
\]
which follows from the generating function for elementary symmetric functions
\cite[Ch.~I, (2.2), p.~19]{Macdonald1995}.  With $a_r=x-r$, this gives
\[
  P_{M-s}(M,x)
  =
  \prod_{r=1}^{M}(x-r)\,
  e_s\!\left(
  \frac1{x-1},\frac1{x-2},\ldots,\frac1{x-M}
  \right).
\]
Finally,
\[
  \prod_{r=1}^{M}(x-r)=\frac{\Gamma(x)}{\Gamma(x-M)}
\]
by the gamma recurrence \cite[(5.5.1)]{DLMF}.  This proves
\eqref{eq:near-diagonal-esym}.  To obtain the partition form, recall the
standard expansion
\[
  e_s=\sum_{\mu\vdash s}
  (-1)^{s-\ell(\mu)}
  z_\mu^{-1}p_\mu
\]
of the elementary symmetric function in terms of power sums, where
$\ell(\mu)=\sum_{j\ge1}m_j$ is the number of parts of $\mu$, and
$p_\mu=\prod_{j\geq1} p_j^{m_j}$; see
\cite[Ch.~I, (2.14)--(2.14'), pp.~24--25]{Macdonald1995}.  For the variables
$1/(x-1),\ldots,1/(x-M)$, the corresponding power sums are
\begin{equation}\label{eq:reciprocal-power-polygamma}
  p_j=\sum_{r=1}^{M}\frac1{(x-r)^j}
  =
  \frac{(-1)^{j-1}}{(j-1)!}
  \left[\psi_{j-1}(x)-\psi_{j-1}(x-M)\right],
\end{equation}
by successive differentiation of the digamma recurrence
\cite[(5.5.2)]{DLMF}; see also \cite[\S5.15]{DLMF}.  Since
\[
  \sum_{j\ge1}(j-1)m_j=s-\ell(\mu),
\]
the signs from the power sums cancel the factor
$(-1)^{s-\ell(\mu)}$, giving \eqref{eq:near-diagonal-partition}.

The first two cases now follow by taking $j=1,2$ in
\eqref{eq:reciprocal-power-polygamma}.  The case $s=1$ gives
\eqref{eq:near-diagonal-first}.  For $s=2$, using $e_2=(p_1^2-p_2)/2$, 
which is the $s=2$ instance of the partition expansion above, gives
\eqref{eq:near-diagonal-second}.
\end{proof}

\subsection{An Application to Barnes Zeta Residues}

The same polynomial structure has a natural interpretation in the theory of
Barnes' multiple zeta function.  Adamchik's finite reduction lemma
\cite[(22)]{Adamchik2005} states that
\begin{equation}\label{adamh1haf}
  \zeta_n(s,x)
  =
  \frac1{(n-1)!}
  \sum_{j=0}^{n-1}
  P^A_{j,n}(x)\zeta(s-j,x),
\end{equation}
where $\zeta(s,x)$ is the Hurwitz zeta function and the coefficients
$P^A_{j,n}(x)$ are defined by \cite[p.~7]{Adamchik2005}
\[
  \prod_{r=1}^{n-1}(y+r-x)
  =
  \sum_{j=0}^{n-1}P^A_{j,n}(x)y^j .
\]
Using the coefficient identification in \eqref{eq:P-product-expansion}, we have
\[
  P^A_{j,m+1}(x)=(-1)^{m-j}P_{m-j}(m,x),
  \qquad 0\le j\le m .
\]
We therefore obtain the following interpretation of the signed generalized
Stirling polynomials.

\begin{proposition}\label{prop:barnes-residue-characterization}
For $0\le k\le m$,
\begin{equation}\label{eq:P-residue-characterization}
  P_k(m,x)
  =
  (-1)^k m!\operatorname*{Res}_{s=m+1-k}\zeta_{m+1}(s,x).
\end{equation}
\end{proposition}

\begin{proof}
Put $n=m+1$ in \eqref{adamh1haf}.  At $s=m+1-k$, the only term with a
pole is the one for which $s-j=1$, namely $j=m-k$.  Since the Hurwitz zeta
function has residue $1$ at $s=1$, we get
\[
  \operatorname*{Res}_{s=m+1-k}\zeta_{m+1}(s,x)
  =
  \frac{P^A_{m-k,m+1}(x)}{m!}
  =
  \frac{(-1)^kP_k(m,x)}{m!},
\]
which is \eqref{eq:P-residue-characterization}.
\end{proof}

The residues of $\zeta_{N}(s,x)$ may also be described by the general Bernoulli-Barnes
residue formula \cite[(3.9)]{Ruijsenaars2000}.  The point of
\eqref{eq:P-residue-characterization} is that it realizes the polynomials
$P_k(m,x)$ as signed Barnes residues.  This identification allows the
identities proved for $P_k(m,x)$ in this work to be read as identities for
$\operatorname*{Res}_{s=m+1-k}\zeta_{m+1}(s,x)$.  For instance, the central
vanishing lemma (Lemma~\ref{lem:central-vanishing}) implies that, for every
odd $k$,
\[
  \operatorname*{Res}_{s=m+1-k}
  \zeta_{m+1}\!\left(s,\frac{m+1}{2}\right)=0.
\]
For the even case, if $1\le q\le n$ and $n-q$ is even, then
\[
  \operatorname*{Res}_{s=q}\zeta_n\!\left(s,\frac n2\right)
  =
  \frac{1}{(n-1)!}
  P_{n-q}\!\left(n-1,\frac n2\right).
\]

\begin{remark}\label{rem:residue-differentiation}
Differentiating Barnes' zeta function with respect to $x$ gives
\[
  \frac{\partial}{\partial x}\zeta_{m+1}(s,x)
  =
  -s\zeta_{m+1}(s+1,x).
\]
Taking the residue at $s=m+1-k$ gives
\begin{equation}\label{eq:residue-differentiation}
  \frac{\dd}{\dd x}
  \operatorname*{Res}_{s=m+1-k}\zeta_{m+1}(s,x)
  =
  -(m+1-k)
  \operatorname*{Res}_{s=m+2-k}\zeta_{m+1}(s,x),
\end{equation}
where the right-hand side is interpreted as $0$ when $k=0$.  By
\eqref{eq:P-residue-characterization}, this is equivalent to
\begin{equation}\label{eq:P-differentiation-from-zeta}
  \frac{\dd}{\dd x}P_k(m,x)
  =
  (m+1-k)P_{k-1}(m,x).
\end{equation}
Replacing $m$ by $m-1$ in \eqref{eq:P-differentiation-from-zeta}
recovers \eqref{deriveqnhaf1}.
\end{remark}

\begin{proposition}[reflection formula]\label{prop:residue-reflection}
For $0\le k\le m$,
\begin{equation}\label{eq:residue-reflection}
  \operatorname*{Res}_{s=m+1-k}
  \zeta_{m+1}(s,m+1-x)
  =
  (-1)^k
  \operatorname*{Res}_{s=m+1-k}\zeta_{m+1}(s,x).
\end{equation}
Equivalently,
\begin{equation}\label{eq:P-reflection}
  P_k(m,m+1-x)=(-1)^kP_k(m,x).
\end{equation}
\end{proposition}

\begin{proof}
Let
\[
  Q_m(y;x)=\prod_{j=1}^{m}(y+j-x).
\]
Then
\begin{equation}\label{usedhere}
  Q_m(y;m+1-x)=(-1)^mQ_m(-y;x).
\end{equation}
Using the expansion \eqref{eq:P-product-expansion},
\[
  Q_m(y;x)=\sum_{k=0}^{m}(-1)^kP_k(m,x)y^{m-k}
\]
in \eqref{usedhere}, we compare coefficients and obtain  \eqref{eq:P-reflection}.  The residue form follows from
\eqref{eq:P-residue-characterization}. 
\end{proof}

\subsection{Cycle-Number Consequences}

We record some consequences for Stirling cycle numbers.  The
closest formulas in \cite[\S26.8]{DLMF}, translated into the notation used
here, are as follows.  Formula \cite[(26.8.27)]{DLMF} is equivalent to
\[
  \stirlingone{n}{n-k}
  =
  \sum_{j=0}^{k}
  (-1)^{k+j}
  \binom{n-1+j}{k+j}
  \binom{n+k}{k-j}
  \left\{\begin{matrix}k+j\\j\end{matrix}\right\}.
\]
This evaluates one cycle number, but it uses Stirling numbers of the second
kind.  Corollary~\ref{cor:central-cycle-band} gives instead a relation among cycle numbers only. 
Formula \cite[(26.8.28)]{DLMF} is equivalent to
\[
  \sum_{k=1}^{n}(-1)^{n-k}\stirlingone{n}{k}=0,
\]
which is a full-row cancellation.  By contrast, \eqref{eq:central-cycle-band}
is a weighted cancellation among
\[
  \stirlingone{n}{n-h},\stirlingone{n}{n-h+1},\ldots,\stirlingone{n}{n}.
\]
Formula \cite[(26.8.30)]{DLMF} is equivalent to
\[
  \sum_{j=k}^{n}
  (-1)^{j-k}\stirlingone{n+1}{j+1}n^{j-k}
  =
  \stirlingone{n}{k}.
\]
This is the closest analogue of Corollary~\ref{cor:integer-shift-cycle}, but
\eqref{eq:integer-shift-cycle} contains the binomial
factor.

The tables in Concrete Mathematics \cite[Tables~250--251]{GrahamKnuthPatashnik1994}
are also related.  Table~250 contains the special value
$\stirlingone{n}{n-1}=\binom n2$, which is the case $h=1$ of
\eqref{eq:central-cycle-band}.  Table~251 gives finite sums involving
Stirling numbers of both kinds.  However, those formulas do not give the
cancellation identities in \eqref{eq:central-cycle-band} or the weighted binomial
transform \eqref{eq:integer-shift-cycle}.  Thus the identities below may be viewed as analogues of classical Stirling-number formulas, although the specific forms \eqref{eq:central-cycle-band} and \eqref{eq:integer-shift-cycle} do not appear to be recorded in \cite[\S26.8]{DLMF} or \cite[Tables~250--251]{GrahamKnuthPatashnik1994}.

\begin{corollary}[cancellation identities]\label{cor:central-cycle-band}
Let $n\ge1$, and let $h$ be an odd integer satisfying $1\le h\le n-1$.  Then
\begin{equation}\label{eq:central-cycle-band}
  \sum_{r=0}^{h}
  (-1)^r2^{h-r}n^r
  \binom{n+r-h-1}{n-h-1}
  \stirlingone{n}{n-h+r}
  =
  0 .
\end{equation}
\end{corollary}

\begin{proof}
Put $m=n-1$, $k=h$, and $x=n/2$ in the explicit formula
\eqref{eq:P-explicit}.  Lemma~\ref{lem:central-vanishing} gives
$P_h(n-1,n/2)=0$ for odd $h$.  Multiplying the resulting identity by
$(-1)^h2^h$ gives \eqref{eq:central-cycle-band}.
\end{proof}

The first case, $h=1$, gives $2\stirlingone{n}{n-1}-n(n-1)\stirlingone{n}{n}=0$.  Thus one recovers the identity
$\stirlingone{n}{n-1}=\binom n2$; see \cite[(26.8.16)]{DLMF} or \cite[Table~250]{GrahamKnuthPatashnik1994}. 

A second family is obtained by evaluating the polynomials at integer points.
These evaluations also give a weighted binomial transform of the tail
$\stirlingone{M+1}{s+1},\ldots,\stirlingone{M+1}{M+1}$
of the $(M+1)$-st row.

\begin{corollary}[integer evaluations and weighted binomial transform]\label{cor:integer-shift-cycle}
For $m\ge1$ and $0\le k\le m$,
\begin{equation}\label{eq:P-integer-evaluations}
  P_k(m,m+1)=\stirlingone{m+1}{m-k+1},
  \qquad
  P_k(m,m)=\stirlingone{m}{m-k}.
\end{equation}
Consequently, for $M\ge1$ and $0\le s\le M$,
\begin{equation}\label{eq:reciprocal-esym-cycle}
  M!\,e_s\!\left(1,\frac12,\ldots,\frac1M\right)
  =\stirlingone{M+1}{s+1},
\end{equation}
and
\begin{equation}\label{eq:integer-shift-cycle}
  \sum_{r=0}^{M-s}
  (-1)^{r-M+s}
  (M+1)^r
  \binom{r+s}{s}
  \stirlingone{M+1}{r+s+1}
  =
  \stirlingone{M+1}{s+1}.
\end{equation}
Equivalently,
\[
  \sum_{j=0}^{M}
  (-1)^{j}
  (M+1)^j
  \binom{j}{s}
  \stirlingone{M+1}{j+1}
  =
  (-1)^M (M+1)^s\stirlingone{M+1}{s+1}.
\]
\end{corollary}

\begin{proof}
Putting $x=0$ in \eqref{eq:P-explicit} gives
$P_k(m,0)=(-1)^k\stirlingone{m+1}{m-k+1}$.
The reflection formula \eqref{eq:P-reflection} therefore gives the first
identity in \eqref{eq:P-integer-evaluations}.  At $x=m$, the recurrence
\eqref{eq:P-step-recurrence} reduces to $P_k(m,m)=P_k(m-1,m)$.
Applying the first identity with $m$ replaced by $m-1$ gives the second
identity for $k<m$, while both sides vanish for $k=m$.
The case $m=1$, $k=0$ follows from $P_0(0,x)=1$.

Now take $m=M$ and $k=M-s$ in the first identity of
\eqref{eq:P-integer-evaluations}, and evaluate \eqref{eq:near-diagonal-esym}
at $x=M+1$.  Since $\Gamma(M+1)/\Gamma(1)=M!$, this gives
\eqref{eq:reciprocal-esym-cycle}.  In particular,
\[
  \prod_{j=1}^{M}\left(1+\frac{x}{j}\right)
  =\sum_{\ell=0}^{M}e_\ell\!\left(1,\frac12,\ldots,\frac1M\right)x^\ell
  =\frac1{M!}\sum_{\ell=0}^{M}\stirlingone{M+1}{\ell+1}x^\ell.
\]
Finally, using \eqref{eq:P-explicit} with $k=M-s$ and $x=M+1$ gives
\eqref{eq:integer-shift-cycle}.
The second form is obtained from \eqref{eq:integer-shift-cycle} by putting
$j=r+s$ and then multiplying both sides by $(-1)^M (M+1)^s$; the terms with $j<s$
vanish because $\binom{j}{s}=0$.
\end{proof}

\begin{remark}
Cereceda defines the second $r$-Stirling polynomial of the first kind by
\cite[p.~7, (23)]{Cereceda2022Axioms}
\[
  \overline R_{m,i}(x)
  =
  \sum_{j=0}^{m-i}
  \binom{i+j}{i}
  \stirlingone{m+1}{i+j+1}x^j .
\]
Comparing this formula with \eqref{eq:P-explicit} gives, for $0\le s\le M$,
\begin{equation}\label{eq:P-Rbar-identification}
  P_{M-s}(M,x)=(-1)^{M-s}\overline R_{M,s}(-x).
\end{equation}
For nonnegative integers $r$, the same article gives
\cite[p.~7]{Cereceda2022Axioms}
\begin{equation}\label{eq:Rbar-r-stirling-specialization}
  \overline R_{m,i}(r)=\rstirlingone{m+r+1}{i+r+1}{r+1},
\end{equation}
where $\rstirlingone{n}{k}{r}$ denotes the $r$-Stirling number of the first
kind, that is, the number of permutations of $\{1,\ldots,n\}$ with $k$ cycles
in which $1,\ldots,r$ lie in distinct cycles.  Broder introduced and studied
these numbers systematically in
\cite{Broder1984}.  The triangular recurrence \cite[(7)]{Broder1984} is
\[
  \rstirlingone{n}{k}{r}
  =
  \rstirlingone{n-1}{k-1}{r}
  +(n-1)\rstirlingone{n-1}{k}{r},
  \qquad n>r,
\]
with $\rstirlingone{r}{r}{r}=1$, $\rstirlingone{r}{k}{r}=0$ for $k\ne r$,
and $\rstirlingone{n}{k}{r}=0$ for $n<r$.  This is the $r$-analogue of
\eqref{stirrlhaf1}.  For $n\ge1$, the cases $r=0$ and $r=1$ reduce to the
ordinary cycle numbers \cite[(5), (6)]{Broder1984}
\[
  \rstirlingone{n}{k}{0}=\stirlingone{n}{k},
  \qquad
  \rstirlingone{n}{k}{1}=\stirlingone{n}{k}.
\]
Also,
\cite[(24)]{Broder1984} gives, for $n\ge r$,
\[
  \prod_{j=r}^{n-1}(z+j)
  =
  \sum_{k=r}^{n}\rstirlingone{n}{k}{r}z^{k-r},
\]
and the exponential generating function is
\cite[Remark 10]{Broder1984}
\[
  \sum_{N=0}^{\infty}\rstirlingone{N+r}{m+r}{r}\frac{z^N}{N!}
  =
  \frac{1}{m!}(1-z)^{-r}
  \left(\log\frac{1}{1-z}\right)^m .
\]
For rook-theoretic elliptic analogues of the Stirling numbers of the first kind
and of their $r$-restricted versions, see Schlosser and Yoo
\cite{SchlosserYoo2017EJC,SchlosserYoo2017AAM}.

Combining
\eqref{eq:P-Rbar-identification} and
\eqref{eq:Rbar-r-stirling-specialization} gives, with $k=m-i$,
\[
  P_k(m,-r)
  =
  (-1)^k\rstirlingone{m+r+1}{m-k+r+1}{r+1},
  \qquad r \in \N \cup \{0\}.
\]
Cereceda also records a negative-integer evaluation of
$\overline R_{m,i}(x)$ at $x=-1$.  In our notation, this follows from
\eqref{eqaftwopnine}.  At $x=1$,
\[
  P_{m-i}(m,1)
  =
  e_{m-i}(0,-1,\ldots,1-m).
\]
Comparing the coefficient of $z^i$ in
\[
  \prod_{\ell=0}^{m-1}(z-\ell)
  =
  \sum_{\ell=0}^{m}e_{m-\ell}(0,-1,\ldots,1-m)z^\ell
  =
  \sum_{\ell=0}^{m}(-1)^{m-\ell}\stirlingone{m}{\ell}z^\ell,
\]
gives
\[
  P_{m-i}(m,1) =(-1)^{m-i}\stirlingone{m}{i}.
\]
Equivalently,
\[
  \sum_{j=0}^{m-i}
  (-1)^j
  \binom{i+j}{i}
  \stirlingone{m+1}{i+j+1}
  =
  \stirlingone{m}{i},
\]
which Cereceda \cite[p.~7]{Cereceda2022Axioms} identifies with
\cite[(6.18)]{GrahamKnuthPatashnik1994}.  Corollary~\ref{cor:integer-shift-cycle}
gives the corresponding evaluation at $x=-(M+1)$.  Indeed, multiplying
\eqref{eq:integer-shift-cycle} by $(-1)^{M-s}$ gives
\[
  \overline R_{M,s}(-(M+1))
  =
  (-1)^{M-s}\stirlingone{M+1}{s+1}.
\]
We did not find this evaluation at $x=-(M+1)$ in Cereceda's work
\cite{Cereceda2022Axioms} or in the Stirling-number tables
\cite[Tables~250--251]{GrahamKnuthPatashnik1994}.  The only explicit
evaluations of $\overline R_{m,i}(x)$ recorded in
\cite[p.~7]{Cereceda2022Axioms} are the nonnegative-integer values $x=r$
and the special value $x=-1$.
\end{remark}

\subsection{The Counting Reduction}

Let $c_j=c_j(N;l_1,\ldots,l_N)$ be the number of times $\chi_j$ occurs in
\eqref{eq:nested-sum}.  In particular, $c_j=0$ for $j<l_1$.  Thus
\begin{equation}\label{eq:coefficient-form}
  S_N(l_1,\ldots,l_N)=\sum_{j=1}^{N}c_j\chi_j .
\end{equation}

The coefficients are obtained by fixing the innermost index.  For example, if
all lower bounds are equal to two and only two summations are present, then
fixing $k_1=j$ gives $j\le k_2\le N$.  Thus $k_2$ may be chosen as
$j,j+1,\ldots,N$, so $\chi_j$ occurs $N-j+1$ times.  Hence
\[
  \sum_{k_2=2}^{N}\sum_{k_1=2}^{k_2}\chi_{k_1}
  =
  \sum_{j=2}^{N}(N-j+1)\chi_j .
\]
For three summations, fixing $k_1=j$ leaves
\[
  j\le k_2\le k_3\le N.
\]
If $k_3=j$, there is one choice for $k_2$; if $k_3=j+1$, there are two
choices; and so on.  The coefficient of $\chi_j$ is therefore
\[
  1+2+\cdots+(N-j+1)
  =
  \binom{N-j+2}{2}.
\]
In general, for a nest of depth $d$ with common lower bound two, fixing
$k_1=j$ leaves the chain
\[
  j\le k_2\le k_3\le\cdots\le k_d\le N,
\]
and hence
\[
  c_j^{(d)}
  =
  \binom{N-j+d-1}{d-1}.
\]
Therefore
\[
  \sum_{k_d=2}^{N}
  \sum_{k_{d-1}=2}^{k_d}
  \cdots
  \sum_{k_2=2}^{k_3}
  \sum_{k_1=2}^{k_2}
  \chi_{k_1}
  =
  \sum_{j=2}^{N}
  \binom{N-j+d-1}{d-1}\chi_j .
\]
For the sums used in $S_N(l_1,\ldots,l_N)$, the depth is $d=N$, and therefore
\begin{equation}\label{thiseq1h}
  c_j=\binom{2N-j-1}{N-1}
      =\binom{2N-j-1}{N-j},
  \qquad 2\le j\le N.
\end{equation}

The same argument applies to any common lower bound.  If
$l_1=\cdots=l_N=m$, then
\begin{equation}\label{eq:common-lower-bound-coefficients}
  c_j=
  \begin{cases}
  \binom{2N-j-1}{N-j},& m\le j\le N,\\
  0,&j<m .
  \end{cases}
\end{equation}
Thus
\begin{equation}\label{eq:common-lower-bound}
  \sum_{k_N=m}^{N}
  \sum_{k_{N-1}=m}^{k_N}
  \cdots
  \sum_{k_2=m}^{k_3}
  \sum_{k_1=m}^{k_2}
  \chi_{k_1}
  =
  \sum_{j=m}^{N}
  \binom{2N-j-1}{N-j}\chi_j .
\end{equation}

\begin{lemma}\label{lem:counting}
With the notation above,
\begin{equation}\label{eq:integral-counting}
  S_N(l_1,\ldots,l_N)
  =
  \int_0^\infty
  \left\{
  \left(\sum_{j=1}^{N}c_j\right)\sech x
  -
  \sum_{j=1}^{N}c_j\sech^j x
  \right\}\frac{\dd x}{x^2}.
\end{equation}
\end{lemma}

\begin{proof}
Equation \eqref{eq:coefficient-form}, together with \eqref{eq:chi-def}, gives
\[
  S_N(l_1,\ldots,l_N)
  =
  \sum_{j=1}^{N}c_j
  \int_0^\infty
  \frac{\sech x-\sech^j x}{x^2}\dd x .
\]
Collecting the coefficients of $\sech x,\sech^2x,\ldots,\sech^Nx$ gives
\eqref{eq:integral-counting}.
\end{proof}

\subsection{Closed Forms for the Integrals }

We now use the closed form $\F_j$ defined in \eqref{eq:Fj}, which is taken from
\cite[(56)]{AbdulsalamMalmsten2025}.
It follows from \eqref{eq:coefficient-form} that
\begin{equation}\label{eq:rhs-counting}
  S_N(l_1,\ldots,l_N)=\sum_{j=1}^{N}c_j\F_j .
\end{equation}
Equivalently, \eqref{eq:yakubovich-sequence} gives
\begin{equation}\label{eq:nested-yakubovich}
  S_N(l_1,\ldots,l_N)
  =\sum_{j=1}^N jc_jh_{j+1}-\frac{4G}{\pi}\sum_{j=1}^N c_j.
\end{equation}
Thus the evaluations of $h_j$ in \cite{Yakubovich2025} can be combined
with the same counting coefficients $c_j$.

Combining \eqref{eq:integral-counting} and \eqref{eq:rhs-counting} gives the
identity
\begin{equation}\label{eq:master-identity}
  \int_0^\infty
  \left\{
  \left(\sum_{j=1}^{N}c_j\right)\sech x
  -
  \sum_{j=1}^{N}c_j\sech^j x
  \right\}\frac{\dd x}{x^2}
  =
  \sum_{j=1}^{N}c_j\F_j .
\end{equation}
Since $\chi_1=0$, the corresponding closed form satisfies $\F_1=0$.  The
apparent singularity at the origin is removable. Indeed, from
$\sech x=1-x^2/2+O(x^4)$ and
$\sech^j x=1-jx^2/2+O(x^4)$, the numerator in
\eqref{eq:master-identity} is
$\frac{x^2}{2}\sum_{j=1}^{N}(j-1)c_j+O(x^4)$ as $x\to 0$.

\begin{proposition}\label{prop:common-lower-bound-identity}
If $1\le m\le N$ and $l_1=\cdots=l_N=m$, then
\begin{equation}\label{eq:common-lower-bound-integral}
  \int_0^\infty
  \left\{
  \binom{2N-m}{N-m}\sech x
  -
  \sum_{j=m}^{N}\binom{2N-j-1}{N-j}\sech^j x
  \right\}\frac{\dd x}{x^2}
  =
  \sum_{j=m}^{N}\binom{2N-j-1}{N-j}\F_j .
\end{equation}
\end{proposition}

\begin{proof}
Summing the coefficients in \eqref{eq:common-lower-bound} and putting
$r=N-j$, the hockey-stick identity \cite[(26.3.7)]{DLMF} gives
\begin{equation}\label{hockeyhh1}
  \sum_{r=0}^{N-m}\binom{N+r-1}{r}
  =
  \binom{2N-m}{N-m}.
\end{equation}
Therefore \eqref{eq:common-lower-bound} gives
\eqref{eq:common-lower-bound-integral}.
\end{proof}

\subsection{A Family with Lower Bounds Equal to Two}

Set $l_1=\cdots=l_N=2$.  This choice removes the trivial $\chi_1$ term.  For
$N=2,3,4,5,6$, the multiplicity vectors $(c_1,\ldots,c_N)$ are respectively
\[
  (0,1),\quad (0,3,1),\quad (0,10,4,1),\quad
  (0,35,15,5,1),\quad (0,126,56,21,6,1).
\]
Substitution into \eqref{eq:master-identity} gives the following identities.
For these evaluations, the closed forms on the right were obtained
from the package command \texttt{NestedFRHS[N,ConstantArray[2,N]]}.
\begin{align}
\int_0^\infty
\left(\sech x-\sech^2x\right)\frac{\dd x}{x^2}
&=\F_2 \notag\\
&=-\frac{4G}{\pi}+\frac{14\zeta(3)}{\pi^2}. \label{eq:A2}
\end{align}

\begin{align}
\int_0^\infty
\left(4\sech x-3\sech^2x-\sech^3x\right)\frac{\dd x}{x^2}
&=3\F_2+\F_3 \notag\\
&=-\frac{14G}{\pi}
  -\frac{\pi}{2}
  +\frac{\psi_3(1/4)}{16\pi^3}
  +\frac{42\zeta(3)}{\pi^2}. \label{eq:A3}
\end{align}
These first two evaluations also follow from
\cite[pp.~31, 37]{Yakubovich2025}, where
\[
  h_3=\frac{7\zeta(3)}{\pi^2},\qquad
  h_4=\frac{16}{\pi}\left(\frac{\beta(4)}{\pi^2}+\frac{G}{24}\right).
\]
Here $\beta$ denotes the ordinary Dirichlet beta function. By
\eqref{eq:yakubovich-sequence},
\[
  \F_2=\frac{14\zeta(3)}{\pi^2}-\frac{4G}{\pi},\qquad
  \F_3=\frac{48\beta(4)}{\pi^3}-\frac{2G}{\pi}.
\]
The coefficients $c_2=3$ and $c_3=1$ therefore give
\[
  3\F_2+\F_3
  =\frac{42\zeta(3)}{\pi^2}+\frac{48\beta(4)}{\pi^3}
   -\frac{14G}{\pi}.
\]
This agrees with \eqref{eq:A3}, since
$\psi_3(1/4)=8\pi^4+768\beta(4)$; in particular, the explicit
$-\pi/2$ in that formula cancels the corresponding polygamma contribution.

\begin{align}
\int_0^\infty
\left(15\sech x-10\sech^2x-4\sech^3x-\sech^4x\right)\frac{\dd x}{x^2}
&=10\F_2+4\F_3+\F_4 \notag\\
&=-\frac{52G}{\pi}-2\pi+\frac{\psi_3(1/4)}{4\pi^3}
  +\frac{448\zeta(3)}{3\pi^2} \notag\\
&\qquad
  +\frac{124\zeta(5)}{\pi^4}. \label{eq:A4}
\end{align}

\begin{align}
\int_0^\infty
\left(\begin{aligned}
&56\sech x-35\sech^2x-15\sech^3x\\
&{}-5\sech^4x-\sech^5x
\end{aligned}\right)\frac{\dd x}{x^2}
&=35\F_2+15\F_3+5\F_4+\F_5 \notag\\
&=-\frac{385G}{2\pi}
  -\frac{33\pi}{4}
  +\frac{95\psi_3(1/4)}{96\pi^3}
  +\frac{\psi_5(1/4)}{768\pi^5} \notag\\
&\qquad
  +\frac{1610\zeta(3)}{3\pi^2}
  +\frac{620\zeta(5)}{\pi^4}. \label{eq:A5}
\end{align}

\begin{align}
\int_0^\infty
\left(\begin{aligned}
&210\sech x-126\sech^2x-56\sech^3x\\
&{}-21\sech^4x-6\sech^5x-\sech^6x
\end{aligned}\right)\frac{\dd x}{x^2}
&=126\F_2+56\F_3+21\F_4 +6\F_5+\F_6 \notag\\
&=-\frac{719G}{\pi}
  -\frac{65\pi}{2}
  +\frac{61\psi_3(1/4)}{16\pi^3}
  +\frac{\psi_5(1/4)}{128\pi^5} \notag\\
&\qquad
  +\frac{29512\zeta(3)}{15\pi^2} +\frac{2728\zeta(5)}{\pi^4}
 +\frac{762\zeta(7)}{\pi^6}. \label{eq:A6}
\end{align}
The general identity for this family is the special case $m=2$ of
\eqref{eq:common-lower-bound-integral}.

\subsection{A Staircase Family}

Now set $(l_1,l_2,\ldots,l_N)=(1,2,\ldots,N)$.  The coefficient of $\chi_1$ is
retained in the count, although its contribution is zero.  For $N=2,3,4,5,6$,
the multiplicity vectors $(c_1,\ldots,c_N)$ are respectively
\begin{equation}\label{multentries}
  (1,1),\quad (2,2,1),\quad (5,5,3,1),\quad
  (14,14,9,4,1),\quad (42,42,28,14,5,1).
\end{equation}
We observe two notable patterns. First, the leading entries
\[
1,2,5,14,42,\ldots
\]
are precisely the Catalan numbers. Second, the sums of the entries in each
vector are likewise Catalan numbers:
\[
1+1=2,\qquad
2+2+1=5,\qquad
5+5+3+1=14,
\]
\[
14+14+9+4+1=42,\qquad
42+42+28+14+5+1=132.
\]
Thus both the first entry and the total multiplicity appear to follow the Catalan sequence.

\begin{proposition}\label{prop:staircase-coefficients}
Let $l_i=i$ for $1\le i\le N$.  Then the coefficient of $\chi_j$ in
$S_N(1,2,\ldots,N)$ is
\begin{equation}\label{eq:staircase-coefficients}
  c_j=
  \frac{j}{N}
  \binom{2N-j-1}{N-1},
  \qquad 1\le j\le N.
\end{equation}
Moreover,
\begin{equation}\label{eq:staircase-catalan-sum}
  \sum_{j=1}^{N}
  \frac{j}{N}
  \binom{2N-j-1}{N-1}
  =
  C_N,
  \qquad
  C_N=\frac1{N+1}\binom{2N}{N}.
\end{equation}
\end{proposition}

\begin{proof}
Fix $j$.  A contribution to the coefficient of $\chi_j$ is the same as a
weakly increasing chain
\[
  j=k_1\le k_2\le\cdots\le k_N=N,
  \qquad k_i\ge i.
\]
Put $n=N-1$ and $h=j-1$.  Such chains are in bijection with lattice paths from
$(0,h)$ to $(n,n)$ using north and east steps and staying weakly above the
line $y=x$: the point $(i-1,k_i-1)$ records the value of the $i$th index, and
consecutive points are joined by north steps followed by one east step.
There are $\binom{2n-h}{n}$ unrestricted paths.  For the paths that first cross
below the line $y=x$, reflect the initial segment through that first crossing;
this gives a bijection with paths from $(h+1,-1)$ to $(n,n)$, of which there
are $\binom{2n-h}{n+1}$.  Hence, by the reflection principle, the number of
admissible paths is
\[
  \binom{2n-h}{n}-\binom{2n-h}{n+1}
  =
  \frac{h+1}{n+1}\binom{2n-h}{n}.
\]
Substituting $n=N-1$ and $h=j-1$ gives
\eqref{eq:staircase-coefficients}.

For the sum, use the same reflected form and put $t=2n-h$.  Then
\[
\begin{aligned}
  \sum_{j=1}^{N}c_j
  &=
  \sum_{h=0}^{n}
  \left(
  \binom{2n-h}{n}-\binom{2n-h}{n+1}
  \right) \\
  &=
  \sum_{t=n}^{2n}
  \left(
  \binom{t}{n}-\binom{t}{n+1}
  \right)
  =
  \binom{2n+1}{n+1}-\binom{2n+1}{n+2}.
\end{aligned}
\]
This last difference equals
$\frac{1}{n+2}\binom{2n+2}{n+1}$, which is $C_N$ because $n=N-1$; see
\cite[\S26.5]{DLMF} for the Catalan lattice-path interpretation.
\end{proof}

Using the multiplicity vectors in \eqref{multentries} in
\eqref{eq:master-identity} gives the following identities.  For these
evaluations, the closed forms on the
right were obtained from the package command \texttt{NestedFRHS[N,lows]} by
taking \texttt{lows=Range[N]}:
\begin{align}
\int_0^\infty
\left(3\sech x-2\sech^2x-\sech^3x\right)\frac{\dd x}{x^2}
&=2\F_2+\F_3 \notag\\
&=-\frac{10G}{\pi}
  -\frac{\pi}{2}
  +\frac{\psi_3(1/4)}{16\pi^3}
  +\frac{28\zeta(3)}{\pi^2}. \label{eq:B3}
\end{align}
\begin{align}
\int_0^\infty
\left(9\sech x-5\sech^2x-3\sech^3x-\sech^4x\right)\frac{\dd x}{x^2}
&=5\F_2+3\F_3+\F_4 \notag\\
&=-\frac{30G}{\pi}
  -\frac{3\pi}{2}
  +\frac{3\psi_3(1/4)}{16\pi^3}
  +\frac{238\zeta(3)}{3\pi^2} \notag\\
&\qquad
  +\frac{124\zeta(5)}{\pi^4}. \label{eq:B4}
\end{align}
\begin{align}
\int_0^\infty
\left(\begin{aligned}
&28\sech x-14\sech^2x-9\sech^3x\\
&{}-4\sech^4x-\sech^5x
\end{aligned}\right)\frac{\dd x}{x^2}
&=14\F_2+9\F_3+4\F_4+\F_5 \notag\\
&=-\frac{185G}{2\pi}
  -\frac{21\pi}{4}
  +\frac{59\psi_3(1/4)}{96\pi^3}
  +\frac{\psi_5(1/4)}{768\pi^5} \notag\\
&\qquad
  +\frac{700\zeta(3)}{3\pi^2}
  +\frac{496\zeta(5)}{\pi^4}. \label{eq:B5}
\end{align}
\begin{align}
\int_0^\infty
\left(\begin{aligned}
&90\sech x-42\sech^2x-28\sech^3x\\
&{}-14\sech^4x-5\sech^5x-\sech^6x
\end{aligned}\right)\frac{\dd x}{x^2}
&=42\F_2+28\F_3+14\F_4  +5\F_5+\F_6 \notag\\
&=-\frac{593G}{2\pi}
  -\frac{71\pi}{4}
  +\frac{193\psi_3(1/4)}{96\pi^3}
  +\frac{5\psi_5(1/4)}{768\pi^5} \notag\\
&\qquad
  +\frac{10892\zeta(3)}{15\pi^2} +\frac{1860\zeta(5)}{\pi^4} 
+ \frac{762\zeta(7)}{\pi^6}. \label{eq:B6}
\end{align}
\begin{proposition}\label{prop:staircase-identity}
For the staircase lower bounds $(l_1,\ldots,l_N)=(1,2,\ldots,N)$, one has
\begin{equation}\label{eq:staircase-general}
  \int_0^\infty
  \left\{
  C_N\sech x
  -
  \sum_{j=1}^{N}
  \frac{j}{N}
  \binom{2N-j-1}{N-1}\sech^j x
  \right\}\frac{\dd x}{x^2}
  =
  \sum_{j=1}^{N}
  \frac{j}{N}
  \binom{2N-j-1}{N-1}\F_j,
\end{equation}
where $C_N$ denotes the $N$th Catalan number.
\end{proposition}

\begin{proof}
Combining \eqref{eq:staircase-coefficients},
\eqref{eq:staircase-catalan-sum}, and \eqref{eq:master-identity} gives the
general staircase identity \eqref{eq:staircase-general}.
\end{proof}

\section{Limiting Identities, Orthogonality, and Further Directions}
\label{sec:further-directions}

The two families treated above illustrate a broader coefficient method.  For
arbitrary lower bounds, the coefficients are still obtained by counting the
possible completions of a nondecreasing chain.  The following recurrence gives
a finite form of this general count.

\begin{proposition}[general lower-bound recurrence]
\label{prop:general-lower-bound-recurrence}
Let $N\ge2$ and let $1\le l_i\le N$ for $1\le i\le N$.  Define functions
$D_i(r)$, for $2\le i\le N$ and $1\le r\le N$, by
\[
  D_N(r)=N-\max(r,l_N)+1
\]
and, recursively,
\[
  D_i(r)=\sum_{s=\max(r,l_i)}^N D_{i+1}(s),
  \qquad 2\le i\le N-1 .
\]
Then the coefficient of $\chi_j$ in $S_N(l_1,\ldots,l_N)$ is
\[
  c_j=
  \begin{cases}
    D_2(j),& j\ge l_1,\\
    0,&j<l_1.
  \end{cases}
\]
\end{proposition}

\begin{proof}
Fix the innermost index $k_1=j$.  If $j<l_1$, then no such term occurs.  Suppose
$j\ge l_1$.  After $k_1$ is fixed, the remaining choices are chains
\[
  j\le k_2\le k_3\le\cdots\le k_N\le N,
  \qquad k_i\ge l_i\quad (2\le i\le N).
\]
The quantity $D_i(r)$ counts the number of possible tails
$k_i,k_{i+1},\ldots,k_N$ under the additional condition $k_i\ge r$.  For
$i=N$ this gives $D_N(r)=N-\max(r,l_N)+1$.  For $2\le i\le N-1$, choosing
$k_i=s$ with $s\ge\max(r,l_i)$ leaves $D_{i+1}(s)$ possible tails.  Summing
over $s$ gives the stated recurrence, and taking $i=2$ and $r=j$ gives the
coefficient of $\chi_j$.
\end{proof}

This recurrence gives an effective way to compute the coefficients for any
fixed lower-bound list. The common-lower-bound and staircase coefficient
formulas \eqref{eq:common-lower-bound-coefficients} and
\eqref{eq:staircase-coefficients} are recovered from this recurrence by taking
$l_1=\cdots=l_N=m$ and $(l_1,\ldots,l_N)=(1,\ldots,N)$, respectively.  The same
coefficients may be encoded in the polynomial
\[
  W(u)=\sum_{j=1}^N c_j u^j,
\]
so that \eqref{eq:master-identity} may be viewed as an identity determined by
the single polynomial $W$.  For the common-lower-bound and staircase families,
this formulation leads to the following normalized limiting identities.

\begin{proposition}[normalized limiting identities]
\label{prop:normalized-limits}
Let $m\ge1$ be fixed and put
\[
  C_{N,m}=\binom{2N-m}{N-m}.
\]
Then
\begin{equation}\label{eq:common-normalized-limit}
  \lim_{N\to\infty}
  \frac{1}{C_{N,m}}
  \sum_{j=m}^{N}
  \binom{2N-j-1}{N-j}\F_j
  =
  \int_0^\infty
  \frac{\sech x-\dfrac{\sech^m x}{2-\sech x}}{x^2}\dd x
  =
  \sum_{r=0}^{\infty}2^{-r-1}\F_{m+r}.
\end{equation}
For the staircase family, with
\[
  C_N=\frac1{N+1}\binom{2N}{N},
\]
one has
\begin{equation}\label{eq:staircase-normalized-limit}
  \lim_{N\to\infty}
  \frac{1}{C_N}
  \sum_{j=1}^{N}
  \frac{j}{N}\binom{2N-j-1}{N-1}\F_j
  =
  \int_0^\infty
  \frac{\sech x-\dfrac{\sech x}{(2-\sech x)^2}}{x^2}\dd x
  =
  \sum_{j=1}^{\infty}\frac{j}{2^{j+1}}\F_j .
\end{equation}
\end{proposition}

\begin{proof}
Throughout the proof, $\F_j$ denotes the closed form for $\chi_j$ introduced in
\eqref{eq:Fj}, with $\F_1=0$.  We first record its growth.  Putting
$x=t/\sqrt j$ in \eqref{eq:chi-def} gives
\[
  \frac{\F_j}{\sqrt j}
  =\int_0^\infty
  \frac{\sech(t/\sqrt j)-\sech^j(t/\sqrt j)}{t^2}\dd t.
\]
For $0\le u\le1$, $0\le u-u^j\le j(1-u)$, and
$1-\sech v\le v^2/2$ for $v\ge0$.  Hence the integrand is bounded by
$\min\{1/2,t^{-2}\}$, independently of $j$.  Since
$\log\sech v=-v^2/2+O(v^4)$ as $v\to0$, dominated convergence gives
\[
  \lim_{j\to\infty}\frac{\F_j}{\sqrt j}
  =\int_0^\infty\frac{1-e^{-t^2/2}}{t^2}\dd t
  =\int_0^\infty e^{-t^2/2}\dd t
  =\sqrt{\frac\pi2},
\]
where the middle equality follows by integration by parts.  In particular,
$\F_j=O(\sqrt j)$.

For the common-lower-bound family, put
\[
  p_{N,j}=\frac1{C_{N,m}}\binom{2N-j-1}{N-j},
  \qquad m\le j\le N,
\]
and put $p_{N,j}=0$ for $1\le j<m$.  For the staircase family, put
\[
  q_{N,j}=\frac1{C_N}\frac jN\binom{2N-j-1}{N-1},
  \qquad 1\le j\le N.
\]
Write $w_{N,j}$ for either family.  The identities \eqref{hockeyhh1} and
\eqref{eq:staircase-catalan-sum} give $\sum_{j=1}^N w_{N,j}=1$.
For each fixed $j$, writing the binomial coefficients as factorials gives
$w_{N,j}\to\delta_j$, where, respectively,
\[
  \delta_j=\begin{cases}0,&1\le j<m,\\2^{m-j-1},&j\ge m,\end{cases}
  \qquad\text{or}\qquad
  \delta_j=\frac{j}{2^{j+1}}.
\]
In both cases, $\sum_{j\ge1}\delta_j=1$.
Moreover,
\[
  \frac{p_{N,j+1}}{p_{N,j}}
  =\frac{N-j}{2N-j-1}\le\frac12,
  \qquad m\le j<N,
\]
and
\[
  \frac{q_{N,j+1}}{q_{N,j}}
  =\frac{j+1}{j}\frac{N-j}{2N-j-1}\le\frac34,
  \qquad 2\le j<N.
\]
For $N\ge2$, $q_{N,1}=q_{N,2}\le1$; also $p_{N,m}\le1$.
Thus, extending the weights by zero for $j>N$, there is a constant $C$
independent of $N$ such that
\[
  0\le w_{N,j}\F_j\le C\left(\frac34\right)^j\F_j,
  \qquad j\ge1.
\]
Here $C$ may depend on the fixed lower bound $m$.  The right-hand side is
summable because $\F_j=O(\sqrt j)$. Since $w_{N,j}\to\delta_j$ for each
fixed $j$, dominated convergence for series (Tannery's theorem) gives
\[
  \lim_{N\to\infty}\sum_{j=1}^N w_{N,j}\F_j
  =\sum_{j=1}^\infty\delta_j\F_j.
\]

It remains to identify the integral. For $0\le t\le1$,
\[
  0\le t-t^j\le(j-1)(1-t).
\]
Consequently,
\[
  0\le
  \sum_{j=1}^{K}\delta_j
  \frac{\sech x-\sech^j x}{x^2}
  \le
  \frac{1-\sech x}{x^2}
  \sum_{j=1}^{\infty}(j-1)\delta_j.
\]
The series on the right converges for both choices of the weights,
and $(1-\sech x)/x^2$ is integrable on $(0,\infty)$.
Dominated convergence therefore allows us to interchange the sum
and the integral. Using $\sum_{j\ge1}\delta_j=1$, we obtain
\[
  \sum_{j=1}^\infty\delta_j\F_j
  =\int_0^\infty
  \frac{\sech x-\sum_{j=1}^\infty\delta_j\sech^j x}{x^2}\dd x.
\]
For $0\le t\le1$, the two generating functions are
\[
  \sum_{j=1}^\infty\delta_jt^j
  =\frac{t^m}{2-t},
  \qquad\text{and}\qquad
  \sum_{j=1}^\infty\delta_jt^j
  =\frac{t}{(2-t)^2},
\]
respectively.  Substituting $t=\sech x$ proves
\eqref{eq:common-normalized-limit} and \eqref{eq:staircase-normalized-limit}.
\end{proof}

\subsection{Orthogonality and the centered polynomials}
\label{subsec:orthogonality-reflection}

The reflection formula \eqref{eq:P-reflection} gives the following parity-cancellation consequence. If we define
\[
  Q_k^{(m)}(t)=P_k\left(m,\frac{m+1}{2}+t\right),
\]
then
\begin{equation}\label{eq:Q-parity-main}
  Q_k^{(m)}(-t)=(-1)^kQ_k^{(m)}(t).
\end{equation}
Thus the centered polynomial $Q_k^{(m)}$ is even when $k$ is even, and odd when
$k$ is odd.

Let $a>0$, and let $W\in L^1([-a,a])$.  If $W$ is even, then
\begin{equation}\label{eq:even-weight-parity-orthogonality}
  \int_{-a}^{a}
  Q_j^{(m)}(t)Q_k^{(m)}(t)W(t)\dd t
  =
  0,
  \qquad\text{whenever }j+k\text{ is odd}.
\end{equation}
If $W$ is odd, then
\begin{equation}\label{eq:odd-weight-parity-orthogonality}
  \int_{-a}^{a}
  Q_j^{(m)}(t)Q_k^{(m)}(t)W(t)\dd t
  =
  0,
  \qquad\text{whenever }j+k\text{ is even}.
\end{equation}
Indeed,
\[
  Q_j^{(m)}(-t)Q_k^{(m)}(-t)
  =
  (-1)^{j+k}Q_j^{(m)}(t)Q_k^{(m)}(t),
\]
so in each of the two cases above the integrand is odd.  This proves the
parity-cancellation relations for every nonnegative $m$.

This, however, does not imply full finite orthogonality.  By full
finite orthogonality we mean that there is a weight, or more generally a
linear functional $\mathcal L$ on polynomials, for which
\[
  Q_0^{(m)},Q_1^{(m)},\ldots,Q_m^{(m)}
\]
are mutually orthogonal and have nonzero squared norms.  That is,
\[
  \mathcal L\!\left(Q_i^{(m)}Q_j^{(m)}\right)=0\quad(i\ne j),
  \qquad
  \mathcal L\!\left(\left(Q_i^{(m)}\right)^2\right)\ne0.
\]
For $m=1,2,3$ this can be done.  In the case $m=1$,
\[
  Q_0^{(1)}=1,\qquad Q_1^{(1)}=t.
\]
The positive symmetric measure
\[
  \frac12(\delta_{-1}+\delta_1)
\]
makes the two polynomials orthogonal; here $\delta_a$ denotes the Dirac measure at $a$. More generally, any symmetric positive measure not concentrated
at the origin works.  For $m=2$ one has
\[
  Q_0^{(2)}=1,\qquad
  Q_1^{(2)}=2t,\qquad
  Q_2^{(2)}=t^2-\frac14.
\]
The positive symmetric measure
\[
  \frac34\delta_0+\frac18(\delta_{-1}+\delta_1)
\]
makes the three polynomials orthogonal.  For $m=3$,
\[
  Q_0^{(3)}=1,\qquad
  Q_1^{(3)}=3t,\qquad
  Q_2^{(3)}=3t^2-1,\qquad
  Q_3^{(3)}=t(t^2-1).
\]
One possible positive symmetric measure is
\[
  \frac{16}{33}(\delta_{-1/2}+\delta_{1/2})
  +\frac{1}{66}(\delta_{-\sqrt3}+\delta_{\sqrt3}).
\]

The situation changes at $m=4$, and the obstruction persists for every
$m\ge4$.  The first case is already clear:
\[
\begin{aligned}
  Q_0^{(4)}&=1,
  &Q_1^{(4)}&=4t,
  &Q_2^{(4)}&=6t^2-\frac52,\\
  Q_3^{(4)}&=t(4t^2-5),
  &Q_4^{(4)}&=\left(t^2-\frac14\right)\left(t^2-\frac94\right).
\end{aligned}
\]
These polynomials satisfy
\begin{equation}\label{whdsat}
  tQ_3^{(4)}
  =
  4Q_4^{(4)}+\frac56Q_2^{(4)}-\frac16Q_0^{(4)}.
\end{equation}
If \((R_n)_{n\ge0}\) is an orthogonal polynomial sequence with nonzero squared
norms, then the standard three-term recurrence theorem for orthogonal
polynomials, also known as Favard's theorem, states that multiplication by $t$ is
tridiagonal; see \cite[Ch.~I, \S4, pp.~18--19]{Chihara1978}:
\begin{equation}\label{favardhh}
  tR_n(t)=A_nR_{n+1}(t)+B_nR_n(t)+C_nR_{n-1}(t).
\end{equation}
Consequently, if the polynomials $Q_0^{(4)}, \ldots, Q_4^{(4)}$ were fully
orthogonal with nonzero squared norms, the term $Q_0^{(4)}$ would not occur
in the expansion \eqref{whdsat} of $tQ_3^{(4)}$.

More generally, the same calculation can be done from the centered product.
Writing
\[
  a_j=\frac{m+1}{2}-j,\qquad 1\le j\le m,
\]
the numbers $a_1,\ldots,a_m$ are symmetric about zero.  Also,
\eqref{eqaftwopnine} gives
\[
  Q_k^{(m)}(t)
  =
  P_k\left(m,\frac{m+1}{2}+t\right)
  =
  e_k(t+a_1,\ldots,t+a_m),
\]
Hence, by the central vanishing lemma (Lemma~\ref{lem:central-vanishing}), with
$n$ replaced by $m+1$, the odd elementary symmetric functions in the $a_j$
vanish. More explicitly,
\[
  e_k(t+a_1,\ldots,t+a_m)
  =
  \sum_{h=0}^{k}\binom{m-h}{k-h}e_h(a_1,\ldots,a_m)t^{k-h},
\]
using the generating function for elementary symmetric functions
\cite[Ch.~I, (2.2), p.~19]{Macdonald1995}.  Equivalently, as in the
combinatorial discussion in the introduction, one may count the terms directly:
after choosing the $h$ entries from which the $a$-terms are taken, the remaining
$k-h$ entries contributing $t$ may be chosen in \(\binom{m-h}{k-h}\) ways.
Therefore $Q_0^{(m)}(t)=1$, and
\[
\begin{aligned}
  Q_2^{(m)}(t)&=\binom m2t^2+e_2,\\
  Q_3^{(m)}(t)&=\binom m3t^3+(m-2)e_2t,\\
  Q_4^{(m)}(t)&=\binom m4t^4+\binom{m-2}{2}e_2t^2+e_4,
\end{aligned}
\]
where $e_2=e_2(a_1,\ldots,a_m)$ and $e_4=e_4(a_1,\ldots,a_m)$.  To evaluate
these constants $e_2$ and $e_4$, set $p_r=\sum_{j=1}^{m}a_j^r$. Since $p_1=p_3=0$, Newton's identities
\cite[Ch.~I, (2.11'), p.~23]{Macdonald1995} give
\[
  e_2=-\frac12p_2,\qquad
  e_4=\frac18(p_2^2-2p_4).
\]
Here
\[
  p_2=\frac{m(m^2-1)}{12},\qquad
  p_4=\frac{m(m^2-1)(3m^2-7)}{240},
\]
and therefore
\[
  e_2=-\frac{m(m^2-1)}{24},\qquad
  e_4=\frac{m(m-1)(m-2)(m-3)(m+1)(5m+7)}{5760}.
\]
Comparing the coefficients of $t^4,t^2$ for $m\ge4$, we obtain
\[
  tQ_3^{(m)}(t)
  =
  \frac{4}{m-3}Q_4^{(m)}(t)
  +\frac{(m-2)(m+1)}{12}Q_2^{(m)}(t)
  -\frac{m(m-1)(m-2)(m+1)}{720}Q_0^{(m)}(t).
\]
The last coefficient is nonzero for every $m\ge4$.  Hence the full finite
orthogonality of
\[
  Q_0^{(m)},Q_1^{(m)},\ldots,Q_m^{(m)}
\]
does not hold for $m\ge4$ by Favard's theorem \eqref{favardhh}.

There is nevertheless a natural infinite orthogonal family near this question.
Centering the product formula \eqref{eq:P-diagonal}, we have
\begin{equation}\label{lasttwoparities}
  q_n(t):=Q_n^{(n)}(t)
  =
  P_n\left(n,\frac{n+1}{2}+t\right)
  =
  \prod_{j=1}^{n}\left(t+\frac{n+1}{2}-j\right).
\end{equation}
Separating the two cases in \eqref{lasttwoparities} yields
\[
  q_{2r}(t)=\prod_{j=1}^{r}\left(t^2-\left(j-\frac12\right)^2\right),
  \qquad
  q_{2r+1}(t)=t\prod_{j=1}^{r}(t^2-j^2).
\]
These centered products also identify the coefficients used in
\cite{Yakubovich2025}. Let $t(n,k)$ denote the signed central factorial
numbers of the first kind, with the convention in
\cite[(1.18), (1.21), (1.22), (1.24)]{Yakubovich2025}.
In particular, \cite[p.~5, (1.18)]{Yakubovich2025} gives
$t(n,n)=1$ for $n\ge0$, so that
\[
t(1,1)=t(2,2)=1.
\]
Furthermore, substituting $x=iu$ in the polynomial identities
\cite[(1.22), (1.24)]{Yakubovich2025} and using the notation
in \cite[(1.21)]{Yakubovich2025}, we obtain
\begin{align*}
  \prod_{j=1}^{m}
  \left(u^2-\left(j-\frac12\right)^2\right)
  &=\sum_{k=0}^{m}t(2m+1,2k+1)u^{2k},
  &&m\ge1,\\
  u\prod_{j=1}^{m-1}(u^2-j^2)
  &=\sum_{k=1}^{m}t(2m,2k)u^{2k-1},
  &&m\ge2.
\end{align*}
Here the second identity also involves replacing the product
length in \cite[(1.22)]{Yakubovich2025} by $m-1$
and multiplying by $u$.

We now compare these products with the expansion in
\eqref{eq:P-product-expansion}, using Lemma~\ref{lem:central-vanishing}
to remove the odd-index terms. For odd $n > 1$, we have
\begin{align*}
\prod_{j=1}^{\left\lfloor\frac{n-1}{2}\right\rfloor}
\left(u^2-\left(j-\frac12\right)^2\right)
&=\prod_{j=1}^{\frac{n-1}{2}}\left(u^2-\left(j-\frac12\right)^2\right)\\
&=\prod_{j=1}^{n-1}\left(u+j-\frac n2\right) =\sum_{r=0}^{\left\lfloor\frac{n-1}{2}\right\rfloor}
P_{2r}\left(n-1,\frac n2\right)u^{n-1-2r}\\
&=\sum_{k=0}^{\frac{n-1}{2}}t(n,2k+1)u^{2k}=\sum_{k=0}^{\frac{n-1}{2}}t(n,n-2k)u^{n-1-2k}.
\end{align*}
For even $n>2$, we similarly have
\begin{align*}
u\prod_{j=1}^{\left\lfloor\frac{n-1}{2}\right\rfloor}(u^2-j^2)
&=u\prod_{j=1}^{\frac{n-2}{2}}(u^2-j^2)\\
&=\prod_{j=1}^{n-1}\left(u+j-\frac n2\right) =\sum_{r=0}^{\left\lfloor\frac{n-1}{2}\right\rfloor}
P_{2r}\left(n-1,\frac n2\right)u^{n-1-2r}\\
&=\sum_{k=1}^{\frac n2}t(n,2k)u^{2k-1} =\sum_{k=0}^{\frac{n-2}{2}}t(n,2k+2)u^{2k+1}\\
&=\sum_{k=0}^{\frac{n-2}{2}}t(n,n-2k)u^{n-2k-1} =\sum_{k=0}^{\left\lfloor\frac{n-1}{2}\right\rfloor}
t(n,n-2k)u^{n-2k-1}.
\end{align*}
In the last sum for odd $n$, we replaced $k$ by $(n-1)/2-k$;
in the even case, we first shifted the index and then replaced
$k$ by $(n-2)/2-k$.

Comparing coefficients of $u^{n-1-2r}$ gives
$P_{2r}(n-1,n/2)=t(n,n-2r)$ for $n\ge3$, $0\le r\le\left\lfloor\frac{n-1}{2}\right\rfloor$. For $n=1,2$, the only admissible index is $r=0$.
Equations \eqref{eq:refh1} and \eqref{eq:P-integer-evaluations}
give $P_0(0,\tfrac12)=1$ and $P_0(1,1)=1$, respectively,
so the initial values above establish the same identification
in these cases. Consequently,
\begin{equation}\label{eq:P-central-factorial}
  P_{2r}\left(n-1,\frac n2\right)=t(n,n-2r),
  \qquad n\ge1,\quad 0\le r\le\left\lfloor\frac{n-1}{2}\right\rfloor.
\end{equation}
Substituting $m=n-1$, $k=2r$, and $x=n/2$ into
\eqref{eq:P-explicit} and using \eqref{eq:P-central-factorial},
we obtain
\begin{equation}\label{eq:central-factorial-stirling}
  t(n,n-2r)
  =
  \sum_{j=0}^{2r}
  (-1)^j\left(\frac n2\right)^j
  \binom{n-1-2r+j}{j}
  \stirlingone{n}{n-2r+j},
  \qquad n\ge1,\quad
  0\le r\le\left\lfloor\frac{n-1}{2}\right\rfloor.
\end{equation}
For example,
\[
  P_2(2,3/2)=t(3,1)=-\frac14,\qquad
  P_2(3,2)=t(4,2)=-1,
\]
\[
  P_2(4,5/2)=t(5,3)=-\frac52,\qquad
  P_4(4,5/2)=t(5,1)=\frac9{16}.
\]
Thus the central factorial coefficients in Yakubovich's integral
evaluations are central values of the polynomials considered here.

A natural classical family to compare $q_{n}(t)$ with is obtained from the
associated Meixner--Pollaczek polynomials.  In the notation of the DLMF, their
connection with the Pollaczek polynomials is given by \cite[(18.35.10)]{DLMF}:
\[
  \mathscr P_n^\lambda(x;\phi,c)
  =
  P_n^{(\lambda)}(\cos\phi;0,x\sin\phi,c),
\]
while their real-line orthogonality and weight are given in
\cite[(18.30.17),(18.30.18)]{DLMF}.  We set
\[
  \lambda=\frac12,\qquad \phi=\frac{\pi}{2},\qquad c=0.
\]
With these substitutions, the weight formula \cite[(18.30.18)]{DLMF}
specializes to
\[
  w^{(1/2)}\left(x,\frac{\pi}{2},0\right)=\sech(\pi x),
\]
using the standard identity
$|\Gamma(1/2+ix)|^2=\pi\sech(\pi x)$.  We shall use the monic normalization
\[
  S_n(t)=\frac{n!}{2^n}\mathscr P_n^{1/2}\left(t;\frac{\pi}{2},0\right),
\]
so that the coefficient of $t^n$ in $S_n(t)$ is one.  In this normalization
the polynomials are characterized by
\begin{equation}\label{recurrencerels}
  S_0(t)=1,\qquad S_1(t)=t,\qquad
  S_{n+1}(t)=tS_n(t)-\frac{n^2}{4}S_{n-1}(t),
\end{equation}
and satisfy
\[
  \int_{-\infty}^{\infty}S_n(t)S_m(t)\sech(\pi t)\dd t
  =
  \frac{(n!)^2}{4^n}\delta_{mn},
\]
where $\delta_{mn}$ is the Kronecker delta.  This is the real-line
Meixner--Pollaczek orthogonality in
\cite[(18.30.17)]{DLMF}, after the above specialization and monic rescaling.
The parity is also visible from the recurrence: $S_0$ is even, $S_1$ is odd,
and multiplication by $t$ switches parity.  Thus $S_n$ has
the same parity as $n$.  The same is clear for $q_n$ from the two product
formulas above. The centered products satisfy
\begin{equation}\label{recurr2hh}
  q_{n+2}(t)
  =
  \left(t^2-\frac{(n+1)^2}{4}\right)q_n(t),
  \qquad q_0(t)=1,\quad q_1(t)=t.
\end{equation}
From the recurrence relations \eqref{recurrencerels} and \eqref{recurr2hh}, one
obtains the exponential generating functions
\[
  \sum_{n\ge0}S_n(t)\frac{z^n}{n!}
  =
  \frac{\exp\!\left(2t\arctan(z/2)\right)}
       {\sqrt{1+z^2/4}},
  \qquad
  \sum_{n\ge0}q_n(t)\frac{z^n}{n!}
  =
  \frac{\exp\!\left(2t\operatorname{arsinh}(z/2)\right)}
       {\sqrt{1+z^2/4}}.
\]
Thus the two families have the same denominator in their exponential
generating functions; the only difference displayed here is the replacement of
$\arctan(z/2)$ by $\operatorname{arsinh}(z/2)$. Since $q_n$ and $S_n$ are both monic of degree $n$ and have the
same parity as $n$, there are unique constants $c_{n,r}$ such that
\begin{equation}\label{eq:MP-central-factorial-comparison}
  S_n(t)
  =
  q_n(t)+\sum_{r=1}^{\lfloor n/2\rfloor}c_{n,r}q_{n-2r}(t).
\end{equation}
In other words, \eqref{eq:MP-central-factorial-comparison} gives a connection 
between the centered products $q_n$ and the orthogonal
Meixner--Pollaczek polynomials $S_n$. The first few cases
are
\[
  S_0=q_0,\qquad
  S_1=q_1,\qquad
  S_2=q_2,\qquad
  S_3=q_3-\frac14q_1,\qquad
  S_4=q_4-q_2-\frac14q_0.
\]
The constants $c_{n,r}$ can be computed by a triangular recurrence.
Write
\[
  S_n(t)=\sum_{r=0}^{\lfloor n/2\rfloor}a_{n,r}t^{n-2r},
  \qquad
  q_n(t)=\sum_{r=0}^{\lfloor n/2\rfloor}b_{n,r}t^{n-2r}.
\]
Comparing coefficients in \eqref{recurrencerels} gives
\begin{equation}\label{eq:MP-monomial-recurrence}
  a_{n+1,r}=a_{n,r}-\frac{n^2}{4}a_{n-1,r-1},\qquad n\ge1,
\end{equation}
with $a_{0,0}=a_{1,0}=1$ and $a_{n,r}=0$ when $r<0$ or $n<2r$.
The product formulas for $q_n$ give
\begin{equation}\label{eq:central-monomial-coefficients}
\begin{aligned}
  b_{2h,r}&=(-1)^r e_r\!\left(\left(\frac12\right)^2,
                 \left(\frac32\right)^2,\ldots,\left(h-\frac12\right)^2\right),\\
  b_{2h+1,r}&=(-1)^r e_r(1^2,2^2,\ldots,h^2),
\end{aligned}
\end{equation}
where $e_r$ denotes the elementary symmetric polynomial of degree $r$,
with $e_0=1$, and the list of variables is empty when $h=0$.
Putting $c_{n,0}=1$, we therefore obtain
\begin{equation}\label{eq:MP-connection-recurrence}
  c_{n,r}=a_{n,r}-\sum_{s=0}^{r-1}c_{n,s}b_{n-2s,r-s},
  \qquad 1\le r\le\lfloor n/2\rfloor.
\end{equation}
Indeed, comparison of the coefficient of $t^{n-2r}$ in
\eqref{eq:MP-central-factorial-comparison} gives
$a_{n,r}=\sum_{s=0}^r c_{n,s}b_{n-2s,r-s}$, and $b_{n-2r,0}=1$.
Thus \eqref{eq:MP-monomial-recurrence}--\eqref{eq:MP-connection-recurrence}
give a finite procedure for all the connection coefficients.  For example,
they give $c_{3,1}=-1/4$, $c_{4,1}=-1$, and $c_{4,2}=-1/4$, as above.
The coefficients $c_{n,r}$ also admit finite-sum representations.  These follow
by putting $x=t^2$ in the product formulas for $q_n$ and using Newton
interpolation.

\begin{proposition}[finite sums for the connection coefficients]
\label{prop:MP-connection-finite-sums}
Let $m\ge0$ and $0\le r\le m$, and put $K=m-r+1$.  With $c_{n,0}=1$,
one has
\begin{equation}\label{eq:MP-connection-even-sum}
  c_{2m,r}
  =\sum_{h=1}^{K}(-1)^{K-h}
  \frac{(2h-1)S_{2m}\!\left(h-\frac12\right)}
       {(K-h)!(K+h-1)!},
\end{equation}
and
\begin{equation}\label{eq:MP-connection-odd-sum}
  c_{2m+1,r}
  =\sum_{h=1}^{K}(-1)^{K-h}
  \frac{2hS_{2m+1}(h)}{(K-h)!(K+h)!}.
\end{equation}
\end{proposition}

\begin{proof}
Since $S_{2m}$ is even and $S_{2m+1}$ is odd, the expressions
\[
  E_m(x)=S_{2m}(\sqrt{x}),\qquad
  O_m(x)=\frac{S_{2m+1}(\sqrt{x})}{\sqrt{x}}
\]
are polynomials in $x$, with the second expression understood after
cancelling the factor $\sqrt{x}$.  Substituting the product formulas for
$q_n$ into \eqref{eq:MP-central-factorial-comparison} gives
\[
\begin{aligned}
  E_m(x)&=\sum_{r=0}^{m}c_{2m,r}
       \prod_{j=1}^{m-r}\left(x-\left(j-\frac12\right)^2\right),\\
  O_m(x)&=\sum_{r=0}^{m}c_{2m+1,r}
       \prod_{j=1}^{m-r}(x-j^2).
\end{aligned}
\]
These are Newton expansions at the nodes $(j-1/2)^2$ and $j^2$,
respectively.  Recall that the coefficient of
$\prod_{j=1}^{K-1}(x-x_j)$ in the Newton expansion of a polynomial $F$ is
its divided difference
\[
  F[x_1,\ldots,x_K]
  =\sum_{h=1}^{K}
  \frac{F(x_h)}
       {\prod_{\substack{1\le\ell\le K\\\ell\ne h}}(x_h-x_\ell)}.
\]
Since $m-r=K-1$, it follows that
\begin{align*}
  c_{2m,r}&=E_m\!\left[\left(\frac12\right)^2,
                    \left(\frac32\right)^2,\ldots,
                    \left(K-\frac12\right)^2\right] = \sum_{h=1}^{K}
 \frac{E_m((h-1/2)^2)}
 {\prod_{\ell\ne h}[(h-\tfrac12)^2-(\ell-\tfrac12)^2]},\\
c_{2m+1,r}&=O_m[1^2,2^2,\ldots,K^2] = \sum_{h=1}^{K}\frac{O_m(h^2)}{\prod_{\ell\ne h}(h^2-\ell^2)}.
\end{align*}
For the even case, factoring the differences of squares gives
\[
\begin{aligned}
  \prod_{\substack{1\le\ell\le K\\\ell\ne h}}
  \left[\left(h-\frac12\right)^2-\left(\ell-\frac12\right)^2\right]
  &=\prod_{\substack{1\le\ell\le K\\\ell\ne h}}(h-\ell)(h+\ell-1)\\
  &=(-1)^{K-h}\frac{(K-h)!(K+h-1)!}{2h-1}.
\end{aligned}
\]
Using $E_m((h-1/2)^2)=S_{2m}(h-1/2)$ gives
\eqref{eq:MP-connection-even-sum}.  Similarly,
\[
  \prod_{\substack{1\le\ell\le K\\\ell\ne h}}(h^2-\ell^2)
  =(-1)^{K-h}\frac{(K-h)!(K+h)!}{2h^2}.
\]
Together with $O_m(h^2)=S_{2m+1}(h)/h$, this gives
\eqref{eq:MP-connection-odd-sum}.
\end{proof}

Thus the connection coefficients may be computed either successively from
\eqref{eq:MP-connection-recurrence} or directly from the finite sums
\eqref{eq:MP-connection-even-sum} and \eqref{eq:MP-connection-odd-sum}.
The centered signed Stirling polynomials give parity cancellation, while
the Meixner--Pollaczek polynomials provide a classical orthogonal family;
the connection coefficients make the relation between these families explicit.

\subsection{Limiting Families and Interpolation}

The limiting identities suggest looking for other lower-bound families whose
coefficient polynomials have stable normalized limits.  There is also a natural
interpolation question behind the quantities $\F_j$.  For $\Re s>0$, define
\[
  \chi(s)=
  \int_0^\infty
  \frac{\sech x-\sech^s x}{x^2}\dd x ,
\]
where $\sech^s x=\exp(s\log(\sech x))$.  Near the origin the numerator is
$(s-1)x^2/2+O(x^4)$, locally uniformly in $s$, while at infinity the integrand
is exponentially decreasing on compact subsets of the half-plane $\Re s>0$.
Thus $\chi(s)$ is holomorphic for $\Re s>0$ and satisfies $\chi(n)=\chi_n$ for
positive integers $n$.  It would be useful to find a corresponding closed form
$\F(s)$ with $\F(n)=\F_n$, and to see whether the common-lower-bound and staircase identities
\eqref{eq:common-lower-bound-integral} and \eqref{eq:staircase-general}
have natural analogues for this interpolating function.

\subsection{Arithmetic Patterns}

The arithmetic pattern in
\eqref{eq:A2}--\eqref{eq:A6} and \eqref{eq:B3}--\eqref{eq:B6} follows
from the evaluations of Yakubovich. Let
\[
  \beta(s)=\sum_{n=0}^{\infty}\frac{(-1)^n}{(2n+1)^s}, \qquad \Re s>0,
\]
denote the ordinary Dirichlet beta function. Equivalently,
$\beta(s)=L(s,\chi_4)$, where $\chi_4$ is the non-principal character
modulo $4$; see \cite[\S25.15(i)]{DLMF}. 
This one-variable function is distinct from the notation $B(z,s)$ used
in \eqref{eq:malmsten-mellin}.

With the signed central factorial convention above,
\cite[(4.29), (4.38), Theorem~4.3]{Yakubovich2025} gives, for $m\ge1$,
\begin{align*}
  h_{2m}&=\frac{(-1)^m4^m}{(2m-1)!}
  \sum_{k=1}^m(-1)^k(2k-1)!t(2m-1,2k-1)
  \frac{\beta(2k)}{\pi^{2k-1}},\\
  h_{2m+1}&=\frac{(-1)^m4^m}{(2m)!}
  \sum_{k=1}^m(-1)^kt(2m,2k)(2k)!(2^{2k+1}-1)
  \frac{\zeta(2k+1)}{(2\pi)^{2k}}.
\end{align*}
Combining these evaluations with \eqref{eq:yakubovich-sequence} gives
\begin{align}
  \F_{2m-1}&=\frac{(-1)^m4^m}{(2m-2)!}
  \sum_{k=1}^m(-1)^k(2k-1)!t(2m-1,2k-1)
  \frac{\beta(2k)}{\pi^{2k-1}}-\frac{4G}{\pi},
  \label{eq:F-beta-finite}\\
  \F_{2m}&=\frac{(-1)^m4^m}{(2m-1)!}
  \sum_{k=1}^m(-1)^kt(2m,2k)(2k)!(2^{2k+1}-1)
  \frac{\zeta(2k+1)}{(2\pi)^{2k}}-\frac{4G}{\pi}.
  \label{eq:F-zeta-finite}
\end{align}
The coefficients may also be written in our notation by
\eqref{eq:P-central-factorial}:
\[
  t(2m-1,2k-1)=P_{2m-2k}\left(2m-2,\frac{2m-1}{2}\right),\qquad
  t(2m,2k)=P_{2m-2k}(2m-1,m).
\]
Since these coefficients are rational, it follows that
\begin{align}
  \F_{2m-1}&\in\operatorname{span}_{\mathbb Q}
  \left\{\frac{\beta(2k)}{\pi^{2k-1}}:1\le k\le m\right\},
  \label{eq:F-odd-span}\\
  \F_{2m}&\in\operatorname{span}_{\mathbb Q}
  \left(\left\{\frac G\pi\right\}\cup
  \left\{\frac{\zeta(2k+1)}{\pi^{2k}}:1\le k\le m\right\}\right).
  \label{eq:F-even-span}
\end{align}
In the odd case, the term $4G/\pi$ is already included by $\beta(2)=G$;
for $m=1$, the formula gives $\F_1=0$.
Consequently, \eqref{eq:rhs-counting} gives the depth-dependent inclusion
\begin{equation}\label{eq:nested-arithmetic-span}
\begin{aligned}
  S_N(l_1,\ldots,l_N)\in\operatorname{span}_{\mathbb Q}\biggl(
  &\left\{\frac{\beta(2k)}{\pi^{2k-1}}:
       1\le k\le\left\lceil\frac N2\right\rceil\right\}\\
  {}\cup{}&\left\{\frac{\zeta(2k+1)}{\pi^{2k}}:
       1\le k\le\left\lfloor\frac N2\right\rfloor\right\}\biggr).
\end{aligned}
\end{equation}
This is a sufficient rational span; we do not assert linear independence or minimality
of its generators.

For the examples of depth at most six displayed above, the span is generated by
\[
  \frac{\beta(2)}{\pi},\quad
  \frac{\beta(4)}{\pi^3},\quad
  \frac{\beta(6)}{\pi^5},\quad
  \frac{\zeta(3)}{\pi^2},\quad
  \frac{\zeta(5)}{\pi^4},\quad
  \frac{\zeta(7)}{\pi^6}.
\]
A separate generator $\pi$ is unnecessary. Indeed, the identities
\[
  \psi_{2r+1}(1/4)=(2r+1)!\zeta(2r+2,1/4),\qquad
  \zeta(s,1/4)=2^{2s-1}\left((1-2^{-s})\zeta(s)+\beta(s)\right)
\]
rewrite the polygamma values arising from our formula \eqref{eq:Fj} for $\F_j$
in terms of Dirichlet beta values and powers of $\pi$. The latter contributions cancel the explicit $\pi$ terms, as
seen after \eqref{eq:A3}. Thus \eqref{eq:F-beta-finite} and \eqref{eq:F-zeta-finite}
yield the rational-span inclusions \eqref{eq:F-odd-span} and
\eqref{eq:F-even-span} for odd and even indices, respectively,
while \eqref{eq:nested-arithmetic-span} describes a rational span
for nested sums of every finite depth.
Weighted chain counts and further ballot-type specializations remain
related directions requiring separate enumeration arguments.

\section{Computational Package}\label{sec:package}

For reproducibility, we have also written an accompanying Wolfram Language
package, \texttt{Malmsten\allowbreak Integral\allowbreak Sequences}, version 1.0
\cite{AbdulsalamMalmstenIntegralSequences2026}.
The package file is intended to be distributed with this article and
implements the identities and reductions used in the examples.  After
loading the package, the principal commands are
\begin{center}
\small
\begin{tabular}{@{}l@{}}
\texttt{SignedStirlingP[k,m,z]},\quad \texttt{MalmstenSequence[n,a,b]},\\
\texttt{ChiSequence[n]},\quad \texttt{FSequence[n]},\\
\texttt{LambdaSequence[n]},\quad \texttt{DeltaSequence[n]},\\
\texttt{NestedChiCoefficients[N,lows]},\quad \texttt{NestedFRHS[N,lows]},\\
\texttt{NestedFSymbolicRHS[N,lows]},\quad \texttt{EvaluateFRHS[expr]},\\
\texttt{NestedChiIntegrandPolynomial[N,lows,x]}.
\end{tabular}
\end{center}
Thus \texttt{MalmstenSequence[n,a,b]} implements
\eqref{eq:malmsten-closed-form}, which is
\cite[Theorem~1]{AbdulsalamMalmsten2025}.  The command
\texttt{SignedStirlingP[k,m,z]} implements \eqref{eq:P-explicit}, which is
\cite[Lemma~1]{AbdulsalamMalmsten2025}.

The command \texttt{ChiSequence[n]} implements the closed form for $\chi_n$
given by \cite[Theorem~2]{AbdulsalamMalmsten2025}; equivalently, in the notation
of the present work, this is the quantity $\F_n$ in \eqref{eq:Fj}.  The package
defines \texttt{LambdaSequence[n]}
from $n\lambda_n=\chi_n+\lambda_1$ in \cite[(9)]{AbdulsalamMalmsten2025}, and
defines \texttt{DeltaSequence[n]} from \eqref{eq:delta-chi-difference}, namely
as \texttt{ChiSequence[n+1]-ChiSequence[n]}.

For a lower-bound list $\texttt{lows}=\{l_1,\ldots,l_N\}$, the nested-sum
commands implement \eqref{eq:nested-sum} as follows:
\begin{center}
\small
\begin{tabular}{@{}ll@{}}
\texttt{NestedChiCoefficients[N,lows]}
  & multiplicities $(c_1,\ldots,c_N)$ in \eqref{eq:coefficient-form},\\
\texttt{NestedChiIntegrandPolynomial[N,lows,x]}
  & polynomial $p(\sech x)$ in \eqref{eq:integral-counting},\\
\texttt{NestedFRHS[N,lows]}
  & closed-form right-hand side \eqref{eq:rhs-counting},\\
\texttt{NestedFSymbolicRHS[N,lows]}
  & symbolic right-hand side in terms of $\F[j]$,\\
\texttt{EvaluateFRHS[expr]}
  & replacement of symbolic terms by \texttt{FSequence[j]}.
\end{tabular}
\end{center}

As a numerical check on the displayed examples, we evaluated the integrals in
\eqref{eq:A2}--\eqref{eq:A6} and \eqref{eq:B3}--\eqref{eq:B6} by
high-precision quadrature and compared them with the corresponding values
returned by \texttt{NestedFRHS}.  In each case, the two values agreed to at
least $70$ decimal digits.

\section*{Competing Interests and Acknowledgements}
This work brings together several stages of the first author's
mathematical training, from his bachelor's thesis on symmetric functions
to his study of analysis at ICTP and his later visit to the University
of Vienna. He is grateful to Praise Adeyemo of the University of Ibadan
for supervising his bachelor's thesis, through which he first studied
symmetric functions, a subject to which the present work returns. 
He is also deeply grateful to Ini Adinya and Gbolahan Lawal of the
University of Ibadan for their guidance, encouragement, and personal
support throughout his undergraduate years.

The first author gratefully acknowledges the
\href{https://mathematik.univie.ac.at/third-mission/zusammenarbeit-mit-dem-globalen-sueden/}
{Vienna African Scholarship} of the University of Vienna. Some of the
results in this article were obtained while he was visiting the Faculty
of Mathematics at the University of Vienna with the support of this
scholarship. He thanks Balázs Szendrői of the University of Vienna for
encouraging him to seek a conceptual meaning of the earlier
Malmsten-type integral evaluations.

He is particularly grateful to Emanuel Carneiro of the Abdus Salam International
Centre for Theoretical Physics (ICTP) for teaching him functional analysis
and for supervising his thesis at ICTP; this training strengthened the
analytical background used in this work. He also acknowledges Carneiro's later
encouragement, which led him to return
to Ripon's formula with greater care and eventually to the combinatorial
explanation of the missing factor.

He thanks the Arbeitsgemeinschaft Diskrete Mathematik
(AGDM, Discrete Mathematics Working Group) at the University of Vienna
for hosting his talk
\href{https://mathematik.univie.ac.at/en/eventsnews/full-news-display/news/signed-generalized-stirling-polynomials-and-generalized-hyperbolic-integrals/}
{``Signed generalized Stirling polynomials and generalized hyperbolic
integrals''} on January 13, 2026. A discussion with Michael J. Schlosser
after that talk led to the present collaboration, from which this article
grew through subsequent exchanges.

The second author acknowledges partial support by the FWF Austrian
Science Fund, grant \href{https://www.doi.org/10.55776/P32305}{10.55776/P32305}. 

Both authors thank the reviewers for their careful reading of the
manuscript and for their helpful comments and suggestions, which
improved its clarity and presentation.

\end{document}